\pdfoutput=1
\documentclass{amsart}
\usepackage[utf8]{inputenc}
\usepackage[english]{babel}
\usepackage{blindtext}
\usepackage{amssymb}
\usepackage[pdftex]{graphicx}
\usepackage{fullpage,amsfonts,nicefrac,amsmath,amsfonts,mathtools,MnSymbol,romannum}
\usepackage{makecell}
\usepackage[shortlabels]{enumitem}
\usepackage{tikz}
\usepackage[cjk]{kotex}
\usepackage{amssymb}
\usepackage{color}
\usepackage{stackengine}
\usepackage{amsthm}
\usepackage{algorithmic, algorithm}
\usepackage{mathtools}

\usepackage{hyperref}
\usepackage[nameinlink]{cleveref}

\usepackage[normalem]{ulem}

\newtheorem{theorem}{Theorem}[section]
\newtheorem{corollary}{Corollary}[theorem]
\newtheorem{lemma}[theorem]{Lemma}
\newtheorem{proposition}[theorem]{Proposition}

\newtheorem{definition}[theorem]{Definition}

\newtheorem{example}[theorem]{Example}

\newtheorem{remark}[theorem]{Remark}

\theoremstyle{definition}
\crefname{thm}{Thm}{Thms}

\newcommand*\circled[1]{\tikz[baseline=(char.base)]{
            \node[shape=circle,draw,inner sep=0.5pt] (char) {\footnotesize #1};}}

\DeclareMathOperator*{\argmin}{arg\,min}

\def\a{\mathbf{a}}
\def\c{\mathbf{c}}
\def\o{\mathbf{o}}
\def\I{\mathbf{I}}
\def\R{\mathbb{R}}

\begin{document}

\title{Filtration learning in exact multi-parameter persistent homology and classification of time-series data}

\author{Keunsu Kim}
\address{
Department of Mathematics \& POSTECH MINDS (Mathematical Institute for Data
 Science), Pohang University of Science and Technology, Pohang 37673, Korea
}
\email{keunsu@postech.ac.kr}

\author{Jae-Hun Jung}
\address{
Department of Mathematics \& POSTECH MINDS (Mathematical Institute for Data
 Science), Pohang University of Science and Technology, Pohang 37673, Korea
}
\email{jung153@postech.ac.kr}


\pagenumbering{arabic}

\begin{abstract}
To analyze the topological properties of the given discrete data, one needs to consider a continuous transform called filtration. Persistent homology serves as a tool to track changes of homology in the filtration. The outcome of the topological analysis of data varies depending on the choice of filtration, making the selection of filtration crucial. Filtration learning is an attempt to find an optimal filtration that minimizes the loss function. Exact Multi-parameter Persistent Homology (EMPH) has been recently proposed, particularly for topological time-series analysis, that utilizes the exact formula of rank invariant instead of calculating it. In this paper, we propose a framework for filtration learning of EMPH. We formulate an optimization problem and propose an algorithm for solving the problem. We then apply the proposed algorithm to several classification problems. Particularly, we derive the exact formula of the gradient of the loss function with respect to the filtration parameters, which makes it possible to directly update the filtration without using automatic differentiation, significantly enhancing the learning process.
\end{abstract}


\maketitle

\section{Introduction}
Topological Data Analysis (TDA) is an analysis method that infers the inherent shape of data from sampled data. Sampled data is typically given discrete, so it is necessary to transform the given discrete data into continuous data, as shown in Figure \ref{fig:filtration}. This transform is referred to as filtration. Persistent homology is a primary tool of TDA, utilizing barcode and persistence diagram that capture the homology changes during the filtration.

\begin{figure}[h]
    \centering
    \includegraphics[width=1 \linewidth]{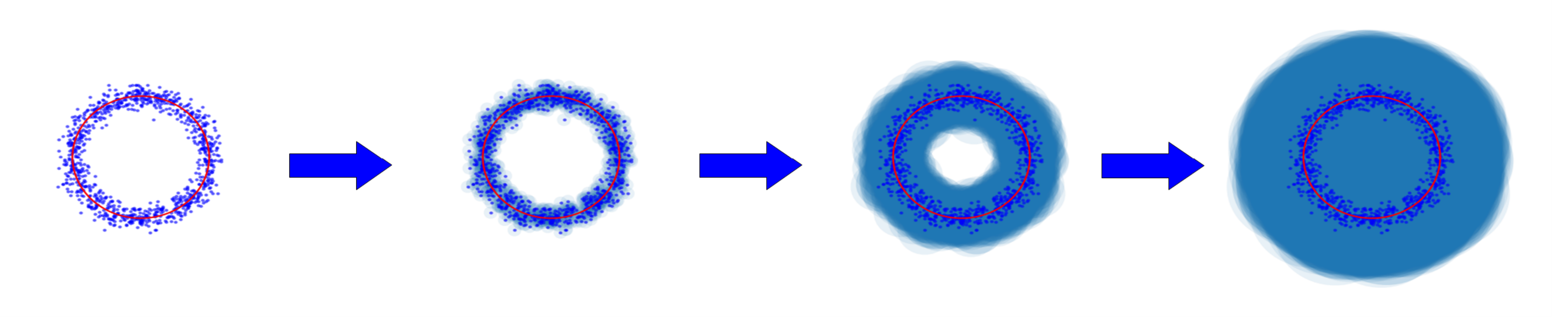}
    \caption{Offset filtration of the point cloud. The figure shows how the given discrete data forms a continuous object through the filtration.}
    \label{fig:filtration}
\end{figure}

\textbf{Topological data analysis and machine learning:} Over the past decade, there has been continuous progress in integrating TDA and Machine Learning (ML). Such integrations are generally composed of the following three steps: 1) For the given data set, the first step considers a filtration corresponding to specific complexes such as Vietoris-Rips complex, cubical complex, DTM filtration etc. and extracts a barcode, which is the summary of the topological and geometric features of the data. 2) The second step transforms the barcode obtained from the first step into vectors using the persistence image, persistence landscape, etc., which are now suitable for ML. 3) The last step takes the transformed vectors and feed them into the ML model for training.

\begin{figure}[h]
    \centering
\includegraphics[width=1 \linewidth]{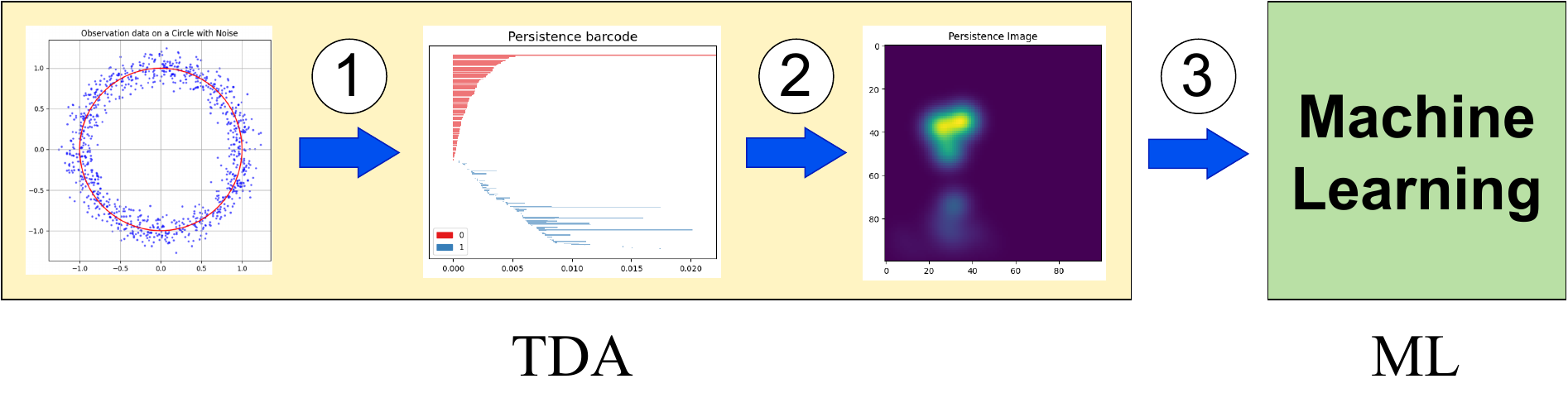}
    \caption{Schematic illustration of the three steps integrating TDA and ML with the given raw data (left). The blue arrow indicates each step.}
\end{figure}

Step \circled{1}: Since topological analysis is intrinsically designed for continuous objects, it is necessary to transform the given discrete data into continuous data, as shown in Figure \ref{fig:filtration}. This transformation is referred to as filtration. Therefore, extracting the topological information of the data depends on the choice of filtration. Commonly used filtrations include the Vietoris-Rips complex, alpha complex, and cubical complex.

Step \circled{2}: The main outcome of persistent homology is represented by barcode, but it lacks a natural vector space structure. Therefore, we need to rearrange or transform the barcode into a vector form in a Hilbert space. This process is called the vectorization of barcode. The persistence landscape \cite{bubenik2015statistical} and persistence image \cite{adams2017persistence} are among the well-known methods for the vectorization of barcode.

Step \circled{3}: This step deals with the main procedure in ML. In this paper, we propose a consistent algorithm of TDA combined with ML that minimizes a task-dependent loss function using the transformed vector calculated in Step \circled{2}.

\textbf{Filtration learning:} Filtration learning is a process in which ML algorithms are involved in Step \circled{1}. For this, it is necessary for each step as a function to be smooth, at least locally Lipschitz. It is also important to ensure convergence when updating each step using backpropagation. One challenge is that the space of barcodes lacks a natural smooth structure. In \cite{bruel2020topology, hofer2020graph, horn2021topological}, it is shown that filtration learning is a smooth process under restricted conditions and it can be applied to ML problems. \cite{leygonie2021framework} proposed a smooth structure on the barcode space and the mentioned studies can be explained within this framework. \cite{carriere2021optimizing} demonstrated the convergence of stochastic gradient descent for loss functions based on persistent homology. \cite{nishikawa2024adaptive} considered filtration learning for the weight of the point in a point cloud.

Before the filtration learning was introduced in the realm of TDA, the design of filtration was based on human understanding and perception, yet filtration learning automatically designs the optimized filtration. TDA with all three steps requires heavy computation as all these steps should be recalculated at every epoch of backpropagation.

\textbf{Exact Multi-parameter Persistent Homology:} Topological time-series analysis differs from traditional statistical time-series analysis. The assumption of topological time-series analysis is that the time-series data are the observations from a dynamical system, and that it involves analyzing the time-series data by examining the topological information of the dynamical system \cite{takens2006detecting, perea2019topological}. In statistics, time series data $\left\{ z_t \right\}_{t \in \mathbb{Z}_{\ge 0}}$ is a realized value of a certain random variable $\left\{ X_t \right\}_{t \in \mathbb{Z}_{\ge 0}}$ \cite{brockwell2002introduction}. Exact Multi-parameter Persistent Homology (EMPH) is one of the topological time-series analysis proposed in \cite{kim2022exact}. Given time-series data, EMPH considers Fourier decomposition and regards each Fourier mode as an uncoupled harmonic oscillator, proposing the analysis based on the Liouville torus. One of the advantages of the proposed method is its ability to provide the exact formula of persistent homology. Traditional topological time-series analysis, such as the one using the sliding window embedding is computationally expensive and the resulting topological measures are difficult to interpret. However, EMPH equipped with the derived exact formula, calculates persistent homology promptly. Furthermore, by constructing multi-parameter persistent homology for the Liouville torus, we can explore various topological inferences through the selection of one-parameter filtration within the multi-parameter space. 

\textbf{Contributions:} In this paper, we propose an optimization problem for filtration learning with EMPH. We provide a detailed algorithm for filtration learning and a proof of its convergence.

Our main contributions can be summarized as follows: 

\begin{enumerate}
    \item We demonstrated the existence of a filtration curve that minimizes the loss function in the multi-parameter space (Theorem \ref{thm:variation principle} in Section \ref{subsec:smoothness}). Also, we formulated a constrained optimization problem to find such a filtration curve and provided the algorithm (\eqref{eq:objective2} in Section \ref{subsec:curved} and Algorithm \ref{alg:filtration main} in Section \ref{sec:examples}).

    \item The vectorization of multi-parameter persistent homology focuses on multiple rays in the multi-parameter space, whereas our proposed method focuses on an optimal curve when constructing EMPH (see Remark \ref{rmk:comparison} in Section \ref{subsec:curved}). This approach can be generalized to an $n$-parameter space ($n \ge 2$) in a calculable manner and yield interpretable results. Additionally, we applied filtration learning to synthetic data and achieved improved accuracy (Examples \ref{Ex:toy1} and \ref{Ex:toy2}). We also compared classification accuracy on a benchmark dataset in Example \ref{example:ucr}. With these experiments, the effectiveness of the proposed method of finding an optimal filtration curve in the multi-parameter space is demonstrated. Additionally, we provided the example of the multi-parameter persistence image and landscape in Appendix \ref{app:vectorization}.

    \item We derived the exact gradient formula of the loss function with respect to the filtration parameters (Corollary \ref{cor:filtration curve} in Section \ref{subsec:curved}). This allows us to update the filtration directly without using automatic differentiation, thereby accelerating the backpropagation in filtration learning. Its efficiency is demonstrated in Example \ref{Ex:time} in Section \ref{sec:examples}.
\end{enumerate}

\textbf{Outline:} Section \ref{sec:Prerequisite} reviews the basics of topological time-series analysis, EMPH theory, a smooth structure in the barcode space and the persistence image. Section \ref{sec:main}, proposes an optimization problem aimed at finding an optimized curve within the EMPH framework. We compare the multi-parameter persistence image and landscape with our method. In Section \ref{sec:examples}, we applied our theory to synthetic and benchmarking data and showed the effectiveness of the proposed mehtod. Section \ref{sec:conclusion} provides a concluding remark and future work.

\section{Background}
\label{sec:Prerequisite}

\begin{subsection}{Basic topological time-series analysis}
\label{subsec:topological time-series analysis}
Basic assumption in topological time-series analysis is that time-series data $\left\{ z_n \right\}_{n \in \mathbb{Z}_{\ge 0}}$ arises from a dynamical system $(\mathcal{M},\phi)$. Here, $\mathcal{M}$ is a smooth manifold that encompasses all possible states, $\phi : \mathcal{M} \rightarrow \mathcal{M}$ is a diffeomorphism describing the state dynamics, and $y : \mathcal{M} \rightarrow \R$ is a smooth function representing measurements. In this setting, our assumption is summarized as follows: $z_n$ is expressed by $y(\phi^n(z_0))$ \cite{takens2006detecting, perea2019topological}. Under this assumption, various characteristics of the dynamical system are reflected in the time-series data. Topological time-series analysis examines the time-series data by focusing on topological features of the dynamical system.

The traditional approach to time-series data analysis in TDA is to translate time-series data into a point cloud via sliding window embedding, then apply the filtration via Vietoris-Rips complex and calculate persistent homology (see Figure \ref{fig:Sliding_pipeline}). A periodicity score, which quantifies the periodicity of the time-series data, is proposed in \cite{perea2015sliding}. This approach utilizes one-dimensional persistence information (the duration of homology) from the trajectory in a dynamical system.
\begin{figure}[h]
    \centering
    \includegraphics[width=1 \linewidth]{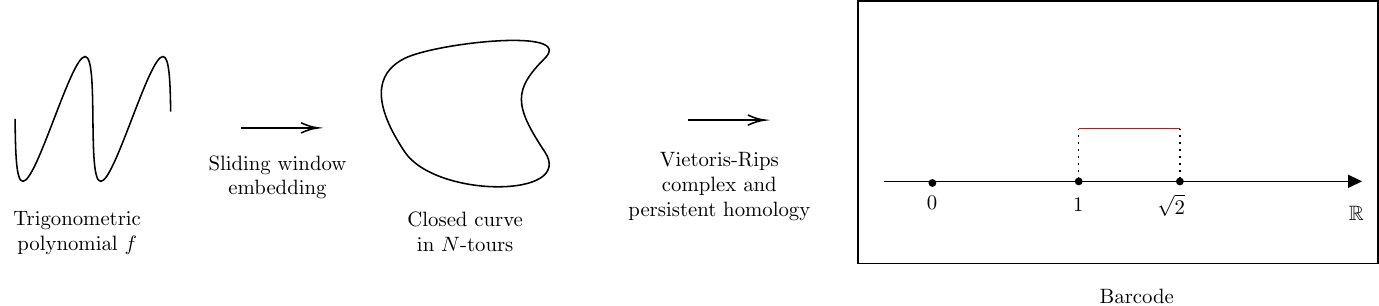}
    \caption{Standard approach of time-series analysis in TDA.}
    \label{fig:Sliding_pipeline}
\end{figure}

\begin{definition}[Sliding window embedding, \cite{perea2015sliding, packard1980geometry}]
For a (periodic) time-series data $f : \R / 2\pi\mathbb{Z} \rightarrow \R$, sliding window embedding $SW_{D,\tau}f$ of $f$ is a point cloud in $\R^{D+1}$ defined by
\begin{equation*}
SW_{D,\tau}f:= \left\{ \begin{bmatrix} f(t) \\ f(t+\tau) \\ \vdots \\ f(t+D\tau) \end{bmatrix} \in \mathbb{R}^{D+1} : t \in \R / 2\pi\mathbb{Z} \right\}.    
\end{equation*}
\end{definition}

\begin{definition}[Vietoris-Rips complex, \cite{adamaszek2017vietoris}]
Given a metric space $(X,d)$, the Vietoris-Rips complex of $(X,d)$ is a filtration of simplicial complexes $\mathcal{R}(X) = \left\{\mathcal{R}_{\epsilon}(X)\right\}_{\epsilon \in \mathbb{R}}$, where
$\mathcal{R}_{\epsilon}(X) := \bigl\{ \left\{ x_{0}, \cdots, x_{l} \right\} \subset P : l \in \mathbb{Z}_{\ge 0} \text { and }\bigr.$ $\bigl. \max\limits_{0 \le i,j \le n} d(x_{i},x_{j}) < \epsilon \bigr\}$.
\end{definition}
Sliding window embedding is motivated by Takens' embedding theorem \cite{takens2006detecting}. This mapping translates time-series data into a point cloud. We then apply the filtration with Vietoris-Rips complex and obtain a barcode. The Vietoris-Rips complex is a popular one-parameter filtration. A one-parameter filtration $\mathcal{X} = \left\{ X_i \right\}_{i \in \R}$ is a family of topological spaces with $X_i \subseteq X_j$ for $i \le j \in \R$.
\begin{proposition}[\cite{latschev2001vietoris}]
\label{prop:rips2}
Let $X$ be a sampled point cloud on a closed Riemannian manifold $\mathcal{M}$, with the Gromov-Hausdorff distance between $X$ and $\mathcal{M}$ being sufficiently close. Then, for a sufficiently small $\epsilon > 0$, $\mathcal{R}_{\epsilon}(X)$ is homotopic to $\mathcal{M}$.
\end{proposition}

Proposition \ref{prop:rips2} provides justification for using the Vietoris-Rips complex of the point cloud $X$ to understand $\mathcal{M}$, but we do not know a priori how to pick such a suitable scale $\epsilon$. Therefore, we consider varying $\epsilon$ and calculate the homology of $\mathcal{R}_{\epsilon}(X)$. This procedure is known as persistent homology.

\begin{definition}[Persistent homology, \cite{oudot2017persistence}]
Given a filtration $\mathcal{X} = \left\{ X_i \right\}_{i \in \R}$, the collection of the $n$-dimensional homology $\left\{ H_n (X_i) \right\}_{i \in \R}$ is called the $n$-dimensional persistent homology of $\mathcal{X}$.   
\end{definition}

Multi-parameter persistence theory was introduced to address the shortcomings of one-parameter persistence theory. For example, one-parameter persistent homology of the Vietoris-Rips complex is known to be sensitive to outliers.
This problem can be rectified by adding one more filtration with the consideration of other data characteristics such as density \cite{carlsson2007theory, lesnick2015interactive}. To use multiple filtrations simultaneously, multi-parameter persistence theory is required \cite{carriere2020multiparameter}. Persistence module is a generalization of persistent homology. Now we define $m$-parameter persistence module.

\begin{definition}[$s$-parameter persistence module \cite{lesnick2015interactive}]
\label{Multidimensiona Persistence Module}
Let $(\mathbb{R}^{s}, \le)$ be an order category, where $(a_1,\cdots,a_s) \le (b_1, \cdots, b_s) \iff a_i \le b_i$ for all $i$, and let $\mathbf{Vect}_{\mathbb{F}}$ be the category of finite dimensional vector spaces over the field $\mathbb{F}$.
The $s$-parameter persistence module is defined by a functor $\mathsf{MP} : \mathbb{R}^{s} \rightarrow \mathbf{Vect}_{\mathbb{F}}$. Persistent homology is a special case of a persistence module, which requires a filtration for its construction.
\end{definition}

The challenge in multi-parameter persistence theory lies in the absence of a complete discrete invariant, which would determine isomorphism classes and represent the characteristics of an object in discrete values, summarizing multi-parameter persistence modules \cite{carlsson2007theory}. As an alternative, we consider an incomplete but practical invariant known as the fibered barcode. This approach involves one-dimensional reductions of multi-parameter persistence modules and is equivalent to the rank invariant \cite{lesnick2015interactive}.

\begin{definition}[Barcode, \cite{carlsson2007theory}]
Let $\mathsf{MP}$ be a one-parameter persistence module. Then $\mathsf{MP} \cong \bigoplus\limits_{j} I(b_j, d_j]$, where $I(b_j, d_j] : \mathbb{R} \rightarrow \mathbf{Vect}_{\mathbb{F}}$ defined by $a \mapsto \mathbb{F}$ if $a \in (b_j,d_j]$ and $a \mapsto \left\{ 0 \right\}$ otherwise. Barcode of $\mathsf{MP}$ is defined by the collection of such $(b_j, d_j]$. It is a complete invariant in the one-parameter persistence module category.   
\end{definition}
A barcode is a summary of persistent homology that simplifies our analysis. It is a complete invariant in the one-parameter persistent module category that determines the isomorphism class.

\begin{definition}[Fibered barcode, \cite{lesnick2015interactive}]
Let $\mathcal{L}$ be a collection of affine lines within $\mathbb{R}^{s}$, each with a nonnegative slope. For any given line $\ell$ in $\mathcal{L}$, we consider a restriction $\mathsf{MP}^{\ell} : \ell \to \mathbf{Vect}_{\mathbb{F}}$. Since this is a one-parameter filtration, we can consider the barcode of this filtration. We refer to $\left\{ \mathsf{bcd}_{n}(\mathsf{MP}^{\ell}) \right\}_{\ell \in \mathcal{L}}$ as the $n$-dimensional fibered barcode of $\mathsf{MP}$.
\end{definition}

\begin{definition}
[Rank invariant, \cite{carlsson2007theory, lesnick2015interactive}]
\label{rank invariant}
Given an $s$-parameter persistence module $\mathsf{MP}$, the rank invariant of $\mathsf{MP}$ is defined as a function $rank(\mathsf{MP}) : \left\{ (t_1,t_2) \in \mathbb{R}^{s} \times \mathbb{R}^{s} : t_1 < t_2 \right\} \rightarrow \mathbb{Z}_{\ge 0}$, which maps $(t_1,t_2) \mapsto rank\left(\mathsf{MP}(t_1) \rightarrow \mathsf{MP}(t_2)\right)$. 
\end{definition}
\end{subsection}

\begin{subsection}{Exact Multi-parameter Persistent Homology}
Recall that the assumption of topological time-series analysis is that time-series data arises from a dynamical system. Therefore, it is natural to explore which properties of time-series data can be further analyzed for the specific dynamical system. A Hamiltonian dynamical system $(\mathcal{M}, \phi_H)$ is a special case of dynamical system in which the trajectory of the state is governed by the Hamilton's equations. Here, $\mathcal{M}$ is a symplectic manifold and $\phi_H : \mathcal{M} \rightarrow \mathcal{M}$ represents the trajectory, called a Hamiltonian diffeomorphism. EMPH considers the Liouville torus of time-series data under the Hamiltonian system of uncoupled one-dimensional harmonic oscillators. It involves constructing multi-parameter persistent homology for the Liouville torus via a multi-parameter sublevel filtration for each Fourier mode and extracting information through rays in the multi-parameter space. Unlike the traditional sliding window embedding method, this approach allows us to derive an exact barcode formula, providing calculable and interpretable inference of the given time-series data.

\begin{definition}[Liouville torus of time-series data, \cite{kim2022exact}]
\label{def:Liouville torus}
Given time-series data $f : \mathbb{R} / 2\pi\mathbb{Z} \rightarrow \mathbb{R}$, we define the ($N$-truncated) Liouville torus $\Psi_{f,N}$ of $f$ as $\Psi_{f,N} := r_1^f \cdot \mathbb{S}^1 \times \cdots \times r_N^f \cdot \mathbb{S}^1$, where $r_L^f := 2\left| \hat{f}(L) \right|$. Here, $\hat{f}(L)$ denotes the $L$th Fourier coefficient of $f$ and $\mathbb{S}^1$ is a unit circle with Euclidean metric. This torus can be represented by the Liouville torus of the system of uncoupled one-dimensional $N$ harmonic oscillators. For the details refer to \cite{kim2022exact}.  From now on, unless otherwise specified, we will abbreviate $\Psi_{f,N}$ as $\Psi_f$.
\end{definition}
Now, we construct multi-parameter persistent homology through a multi-parameter sublevel filtration for each Fourier mode. For each parameter $(\epsilon_1, \cdots, \epsilon_N) \in \R^N$ in the multi-parameter space, we designate the product of Vietoris-Rips complexes $\mathcal{R}_{\epsilon_1}\left(r_1^f \cdot \mathbb{S}^1\right) \times \cdots \times \mathcal{R}_{\epsilon_N}\left(r_N^f \cdot \mathbb{S}^1\right)$. Then, we consider the one-parameter reduction of the constructed multi-parameter persistent homology along a ray $\ell$. Since this reduction involves a one-parameter filtration, we can define the barcode within this one-parameter filtration. That is, we consider the rank invariant of the multi-parameter persistent homology. 
\begin{definition}{\cite{kim2022exact}}
\label{Def:multi-parameter}
Define a multi-parameter filtration $\mathcal{R}\left(\Psi_{f}\right) : \mathbb{R}^{N} \rightarrow \mathbf{Simp}$ by $\boldsymbol{\epsilon} = (\epsilon_1,\cdots,\epsilon_N) \mapsto \mathcal{R}_{\epsilon_1}\left(r_1^f \cdot \mathbb{S}^1\right) \times \cdots \times \mathcal{R}_{\epsilon_N}\left(r_N^f \cdot \mathbb{S}^1\right),$ where $\mathbf{Simp}$ is the category of simplicial complexes. Now, let a ray with direction vector $\mathbf{a}$ be $\ell(t) = \sqrt{N} t \cdot {\mathbf{a} \over \lVert \mathbf{a} \rVert}$ in the multi-parameter space, and define a one-parameter reduced filtration $\mathcal{R}^{\ell}\left(\Psi_{f}\right) : \mathbb{R} \rightarrow \mathbf{Simp}$ specified by $\mathcal{R}^{\ell}_{t} \left(\Psi_{f}\right) = \mathcal{R}_{\sqrt{N} t \cdot {a_1 \over \lVert \mathbf{a} \rVert}}\left(r_1^f \cdot \mathbb{S}^1\right) \times \cdots \times \mathcal{R}_{\sqrt{N} t \cdot {a_N \over \lVert \mathbf{a} \rVert}}\left(r_N^f \cdot \mathbb{S}^1\right)$. We refer to the barcode $\mathsf{bcd}_{n}^{\mathcal{R},\ell}(\Psi_{f})$ of this filtration as the EMPH of $f$.
\end{definition}

\begin{example}
\label{ex:multi-filtration}
For the time-series data of $f(t) = \cos t + {1\over 2} \cos 2t$, its Liouville torus is $\Psi_f = \mathbb{S}^1 \times {1 \over 2} \cdot \mathbb{S}^1$. We assign each simplicial complex $\mathcal{R}_{\epsilon_1} (\mathbb{S}^1) \times \mathcal{R}_{\epsilon_2} ({1 \over 2} \cdot \mathbb{S}^1)$ to a point $(\epsilon_1,\epsilon_2) \in \R^2$ in the multi-parameter filtration. Figure \ref{fig:multi-filtration} serves as a supplement to this example.

\begin{figure}[h]
    \centering
    \includegraphics[width=0.6\linewidth]{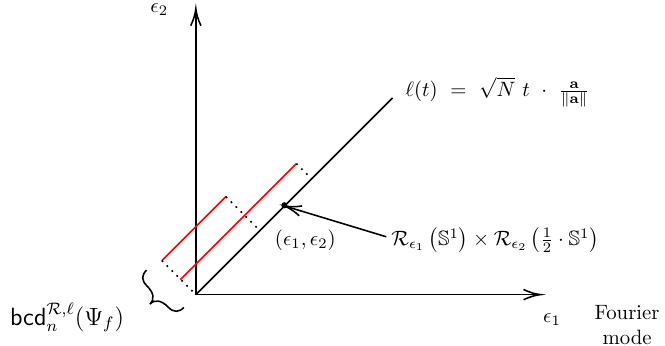}
    \caption{Multi-parameter filtration defined in Example \ref{ex:multi-filtration}.}
    \label{fig:multi-filtration}
\end{figure}    
\end{example}

\begin{theorem}{\cite{kim2022exact}}
\label{thm:EMPH}
If a  ray $\ell$ in the filtration space has the  direction vector $\mathbf{a} = (a_{1},\cdots , a_{N})>0$, then the $n$-dimensional barcode is given by the following:  
\begin{equation}
\mathsf{bcd}_{n}^{\mathcal{R},\ell}(\Psi_{f}) = \left\{ J_{1}^{n_1, \ell}\bigcap \cdots \bigcap J_{N}^{n_N, \ell} :  J_{L}^{n_L, \ell} \in  \mathsf{bcd}_{n_L}^{\mathcal{R},\ell} 	\left( r_L^f \cdot \mathbb{S}^1 \right)  \ \text{and} \ \sum\limits_{L=1}^{N} n_L = n \right\},
\label{eq:EMPH}
\end{equation}
where
\begin{equation}
J_{L}^{n_L, \ell} = \begin{cases} \left(0, \infty\right), & \mbox{if }n_L =0 \\ \left({2r_L^f \sin\left(\pi {k \over 2k+1}\right)  \over \sqrt{N}a_L / \lVert \mathbf{a} \rVert}  , {2r_L^f \sin\left(\pi {k+1 \over 2k+3}\right)  \over \sqrt{N}a_L / \lVert \mathbf{a} \rVert}\right], & \mbox{if }n_L =2k+1 \\ \ \emptyset, & \mbox{otherwise}
\end{cases}.
\label{eq:barcode in EMPH}
\end{equation}
\end{theorem}

With this theorem, we can discern which information of the given time-series data is encoded in the barcode of EMPH. Table \ref{table:relation} shows the information of time-series data encoded in the barcode.

\begin{table}[h]
    \centering
    \includegraphics[width=1 \linewidth]{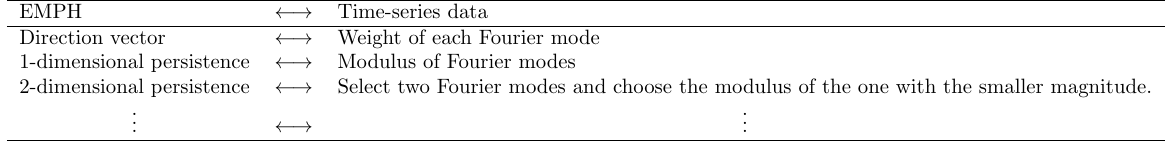}
    \caption{Relationship between EMPH and time-series data.}
    \label{table:relation}
\end{table}
The EMPH consists of the following pipeline \cite{kim2022exact}:

$$\begin{array}{ccccc}
f & \longrightarrow & \mathcal{R}_{\epsilon_1}(r_1^f \cdot \mathbb{S}^1) \times \cdots \times \mathcal{R}_{\epsilon_N}(r_N^f \cdot \mathbb{S}^1)  & \longrightarrow & \mathsf{bcd}_{n}^{\mathcal{R},\ell}(\Psi_{f}) \\
\text{Time-series data} &&  \begin{array}{c} \text{Multi-parameter filtration} \\ \text{of the Liouville torus} \end{array} && \text{Barcode}
\end{array}$$

In \cite{kim2022exact}, the authors utilized EMPH within the contexts of both unsupervised and supervised learning problems. For the unsupervised learning, it was demonstrated that the choice of ray provides a variety of inferences. In this paper, our goal is to develop a method of finding an optimal ray. We utilize the gradient descent method to minimize the loss function and update the ray. To achieve this, it is essential to introduce a smoothness of the update process.

\end{subsection}

\subsection{Smooth structure in the space of Barcodes}
The motivation for updating the filtration in persistent homology, rather than adhering to a fixed filtration, is to more accurately capture the characteristics of the data within the filtration process. Therefore, we assign parameters to filtrations and attempt to find a filtration which minimizes the loss function. The challenging aspect is that persistent homology is summarized by a barcode, but the space of barcodes lacks a natural smooth manifold structure. This difficulty hinders the application of gradient descent. Diffeology is a conceptual generalization of smooth charts that enables us to define smoothness in the space of barcodes. Motivated by diffeology theory, the following definitions are introduced in \cite{leygonie2021framework}.
\begin{definition}{\cite{leygonie2021framework}}
\label{def:smooth}
Consider smooth manifolds $\mathcal{M}$ and $\mathcal{N}$, and denote the space of barcodes by $\mathsf{Bar}$. For a map $B : \mathcal{M} \rightarrow \mathsf{Bar}$, we define it as $C^k$-differentiable at $x \in \mathcal{M}$ if there exists a neighborhood $U$ of $x$ for $p, q \in \mathbb{Z}_{\ge 0}$ and a $C^k$-differentiable map $\tilde{B} : \mathcal{M} \rightarrow \mathbb{R}^{2p} \times \mathbb{R}^{q}$ such that $B = Q_{p,q} \circ \tilde{B}$ on $U$. Here, $p$ and $q$ represent the number of finite and infinite bars, respectively, and $Q_{p,q} (b_1,d_1, \cdots, b_p,d_p, v_1, \cdots, v_q) = \left\{ (b_1, d_1], \cdots, (b_p, d_p], \cdots, (v_1, \infty) , \cdots, (v_q, \infty) \right\}$. It is important to note that both spaces, $\mathcal{M}$ and $\mathbb{R}^{2p} \times \mathbb{R}^{q}$, are smooth manifolds, each equipped with a natural smooth structure. A map $V : \mathsf{Bar} \rightarrow \mathcal{N}$ is called $C^k$-differentiable at $\mathcal{B} \in \mathsf{Bar}$ if for all $p,q \in \mathbb{Z}_{\ge 0}$ and $\tilde{\mathcal{B}} \in \R^{2p} \times \R^{q}$ such that $Q_{p,q}(\tilde{\mathcal{B}}) = \mathcal{B}$, the map $\tilde{V} = V \circ Q_{p,q} : \R^{2p} \times \R^{q} \rightarrow \mathcal{N}$ is a $C^k$-differentiable map on a neighborhood of $\tilde{\mathcal{B}}$. Figure \ref{fig:smooth} is a schematic illustration of this definition.
\end{definition}

\begin{figure}[h]
    \centering
    \includegraphics[width=1\linewidth]{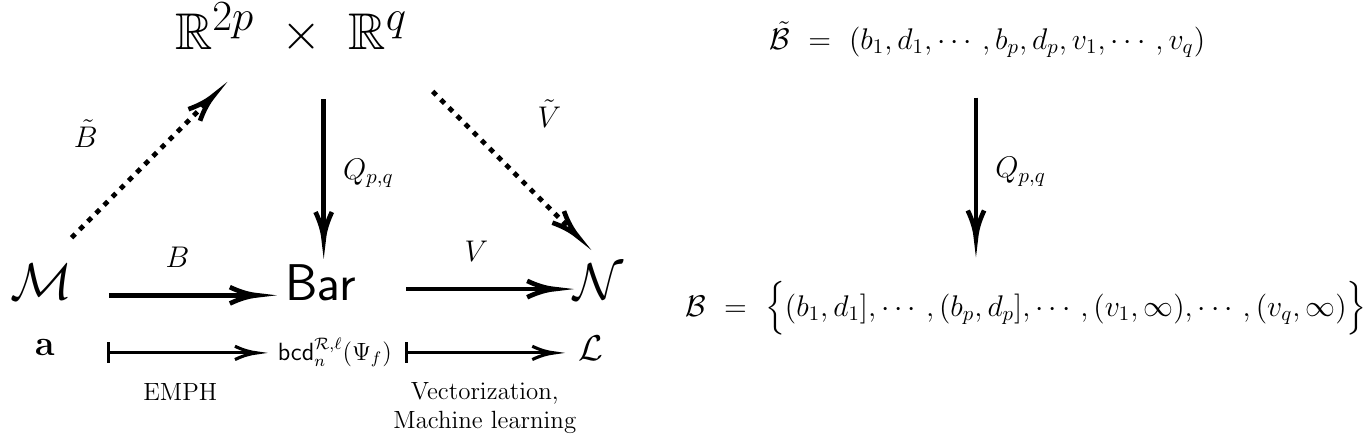}
    \caption{Schematic illustration of smooth structure on the space of barcodes.}
    \label{fig:smooth}
\end{figure}



In order to utilize machine learning, it is necessary to introduce such a differential structure to the barcode space when searching for parameters that minimize the loss functions. Our goal is to ensure that when $V$ and $B$ are smooth defined in Definition \ref{def:smooth}, $V \circ B : \mathcal{M} \rightarrow \mathcal{N}$ is also smooth within the smooth manifold category.

\begin{proposition}{\cite{leygonie2021framework}}
Suppose that $B : \mathcal{M} \rightarrow \mathsf{Bar}$ and $V : \mathsf{Bar} \rightarrow \mathcal{N}$ are $C^k$-differentiable. Then $V \circ B : \mathcal{M} \rightarrow \mathcal{N}$ is a well-defined (the derivative of $V \circ B$ is independent of $\tilde{B}$ and $\tilde{D}$) $C^k$-differentiable function in the smooth manifold category.
\end{proposition}


\subsection{Persistence Image}
Note that the space of barcodes lacks not only a natural differential structure but also a vector space structure. This limitation makes it difficult to integrate TDA and ML. To overcome this limitation, several studies have been conducted to develop smooth mappings from barcodes to a Hilbert space; among these, the persistence landscape \cite{bubenik2015statistical} and persistence image \cite{adams2017persistence} are well-known methods. In this study, we use vectorization with the persistence image because it is proven to be a smooth map in \cite{leygonie2021framework}, allowing us to prove the smoothness of our workflow.

\begin{definition}{\cite{adams2017persistence}}
Given a barcode $\mathcal{B}$, its persistence image is an $r^2$-dimensional vector $\I_{\mathcal{B}} = \bigl( I_{\mathcal{B}} (x_1, y_1), \cdots,$ $I_{\mathcal{B}} (x_{r^2}, y_{r^2}) \bigr) \in \mathbb{R}^{r^2}$, where
\begin{equation}
I_{\mathcal{B}}(x,y) = \sum\limits_{(b_i,d_i] \in \mathcal{B}} \omega(d_i - b_i)g_{b_i, d_i}(x,y),
\label{eq:persistence image}
\end{equation}
and $\omega : \mathbb{R} \rightarrow \mathbb{R}$ is a weight function and $g_{b_i, d_i}(x,y)$ is defined as ${1 \over 2\pi\sigma^2} e^{-{{\left[ (x-b_i)^2 + (y - (d_i - b_i))^2 \right]} \over 2\sigma^2}}$. Here, $(x_j, y_j)$ represents a subdivided point within the given square. From now on, unless otherwise specified, we will abbreviate $\I_{\mathcal{B}}$ as $\I$.
\end{definition}

In \cite{adams2017persistence}, the persistence image is defined as the integration of $g_{b_i, d_i}$. However, for practical purpose, the above discrete form is usualy used, e.g. in Gudhi \cite{gudhi:urm}. The following propositions highlight the properties of the persistence image.

\begin{proposition}[\cite{adams2017persistence}]
\label{prop:persistence image}
Let $W_1$ be the $1$-Wasserstein distance on the space of barcodes. The persistence image $\I : (\mathsf{Bar}, W_1) \rightarrow \mathbb{R}^n$ is a Lipschitz function.  
\end{proposition}

Given that for barcodes $\mathcal{B}_1$ and $\mathcal{B}_2$, $W_p(\mathcal{B}_1, \mathcal{B}_2) \le W_q(\mathcal{B}_1, \mathcal{B}_2)$ when $p \ge q$, it follows that Proposition \ref{prop:persistence image} also holds for the bottleneck distance $(\mathsf{Bar}, d_B)$ \cite{skraba2020wasserstein}.

\begin{proposition}{\cite{leygonie2021framework}}
If the weight function $\omega$ is $C^k$-differentiable, then the persistence image is also $C^k$-differentiable.    
\end{proposition}


\section{Filtration learning within Exact Multi-parameter Persistent Homology}
\label{sec:main}
In this section, we propose a novel filtration learning with EMPH, which is summarized as an optimization problem. We first demonstrate the existence of an optimal filtration ray and, more generally, an optimal filtration curve in the multi-parameter space. Here, `optimal filtration' means the filtration that minimizes the given loss function. We address the problem of finding an optimal filtration ray in Section \ref{subsec:filtration ray}, and then we generalize the problem that addresses an optimal filtration curve in Section \ref{subsec:curved}.

\subsection{Filtration ray learning}
\label{subsec:filtration ray}
In this section, we discuss how to find an optimal filtration ray. We refer this process to as filtration ray learning, as defined in \eqref{eq:objective}. Figure \ref{fig:Pipeline_of_optimization2} summarizes our model; given time-series data (left), we apply EMPH to obtain a barcode, then transform it into a persistence image (PI) and input the PI into the neural network. In the multi-parameter space, if the direction vector of a ray changes, EMPH changes. As a result, the obtained PI changes. Consequently, the input to the neural network changes as well. With this process, we aim to identify a direction vector that minimizes the loss function.

\begin{figure}[h]
    \centering
    \includegraphics[width=1\linewidth]{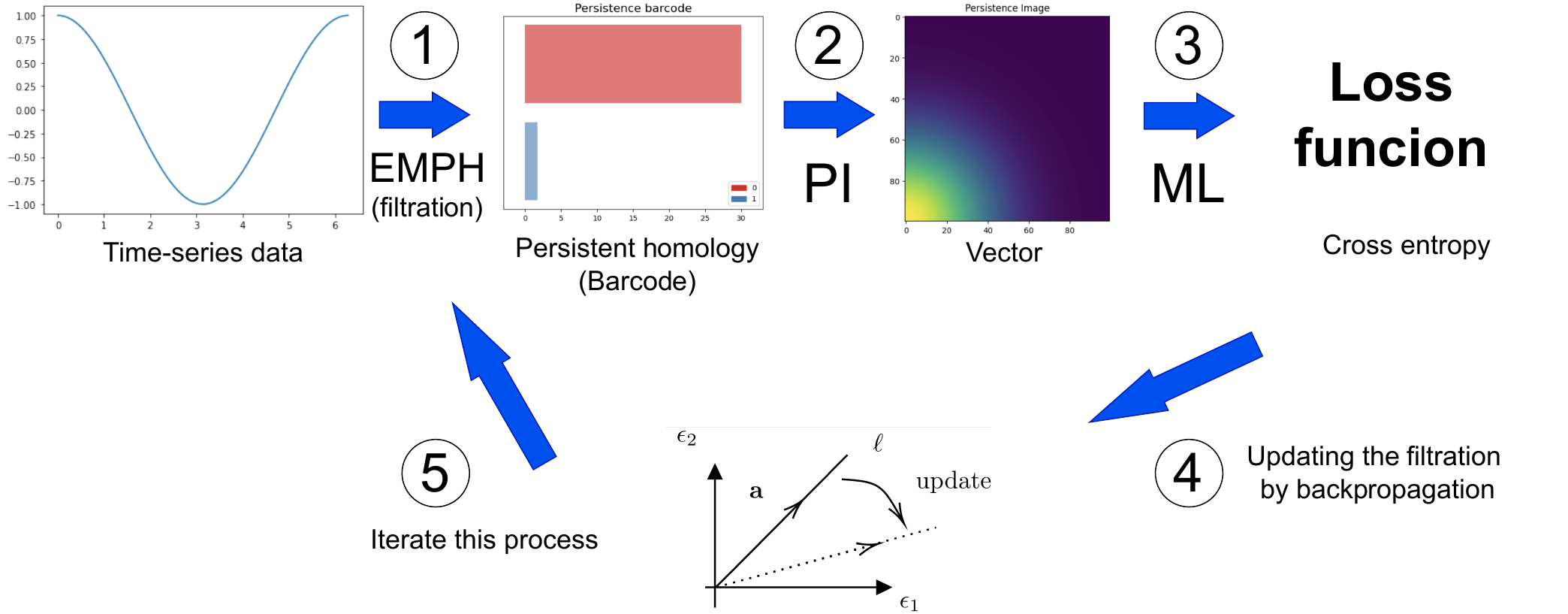}
    \caption{Summary of the proposed filtration learning model.}
    \label{fig:Pipeline_of_optimization2}
\end{figure}

We provide the notations in Table \ref{table:definition} used for the further analysis below. 

\begin{table}[h]
\includegraphics[width=\linewidth]{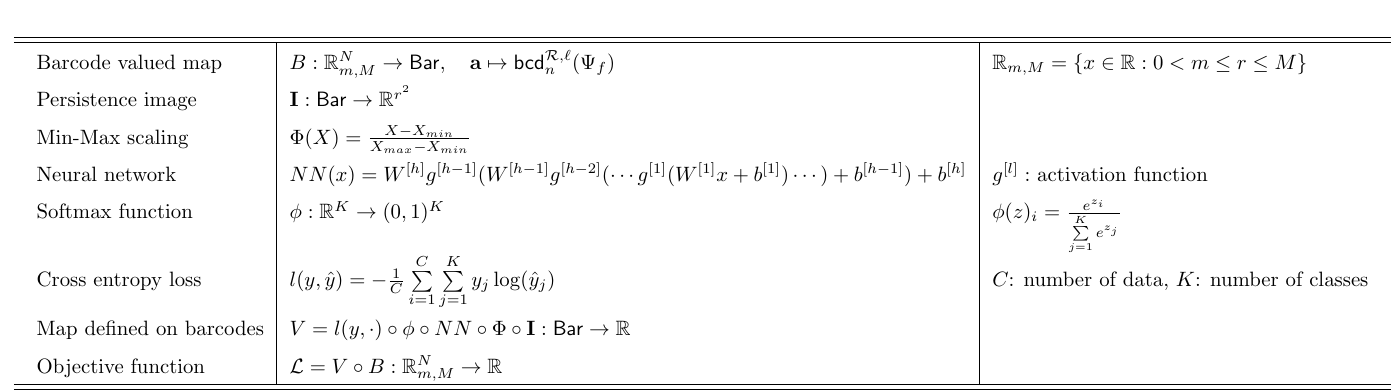}
\caption{Definition of functions used in filtration ray learning.}
\label{table:definition}
\end{table}

Our problem is summarized as the following constrained optimization problem:
\begin{equation}
\label{eq:objective}
\argmin_{\substack{\mathbf{W},\mathbf{b}, \\ \a \in \mathbb{R}_{m,M}^N}} \mathcal{L}(\mathbf{W},\mathbf{b}, \a),
\end{equation}
where $\mathbf{W}$ and $\mathbf{b}$ represent the weight and bias in the neural network, respectively, and $\a$ is the direction vector of a ray in the multi-parameter space, which is constrained to $\mathbb{R}_{m,M}^N$. The reason for restricting the filtration parameter space to $\mathbb{R}_{m,M}^N$ instead of $\mathbb{R}_{> 0}^N$ is to ensure the existence of the optimal parameter $\a$ (Theorem \ref{thm:variation principle}). Note that if we pick an excessively small $m$, it permits small entries of $\a \in \mathbb{R}_{m,M}^N$, which can lead to get excessively large persistence according to \eqref{eq:barcode in EMPH}. Therefore, in the gradient descent to find a solution to \eqref{eq:objective}, it would be unstable, even with a small learning rate. We need to choose $m$ appropriately.

What is remarkable about the proposed optimization problem is that the loss function includes $\a$ as a variable. We want to find a minimizer $\a$ which minimizes the loss function. We demonstrate the smoothness of $\mathcal{L}$ and explain how to solve the optimization problem in the following section.

\begin{remark}
We applied Min-Max scaling $\Phi$ to the persistence image. At each epoch, the persistence image is updated, with its scale being changed, but Min-Max scaling yields better results according to our numerical experiments.
\end{remark}

\subsubsection{Smoothness}
\label{subsec:smoothness}
In the filtration learning of EMPH, gradient descent is used to minimize the loss function. To show convergence, we first demonstrate the smoothness of our model for one-dimensional EMPH.
\begin{lemma} 
\label{lemma:smooth}
Let $B : \mathbb{R}_{> 0}^N \rightarrow \mathsf{Bar}$ be a one-dimensional $(i.e., n=1)$ barcode valued map defined by $\mathbf{a} \mapsto \mathsf{bcd}_{1}^{\mathcal{R},\ell}(\Psi_{f})$. Then $B$ is a smooth map.
\end{lemma}

\begin{proof}
Define $\tilde{B} : \mathbb{R}_{> 0}^N \rightarrow \mathbb{R}^{2N}$ by $\tilde{B}(\a) = \left(b_1(\a), d_1(\a), b_2(\a), d_2(\a), \cdots, b_N(\a), d_N(\a)\right)$, where $b_L, d_L : \mathbb{R}_{> 0}^N \rightarrow \mathbb{R}$ with $b_L(\mathbf{a}) = 0$ and $d_L(\mathbf{a}) = {\sqrt{3} r_L^f   \over \sqrt{N}a_L / \lVert \mathbf{a} \rVert}$. Then $\tilde{B}$ is clearly smooth map for every $\mathbf{a} \in \mathbb{R}_{> 0}^N$ and $B(\a) = \mathsf{bcd}_{1}^{\mathcal{R},\ell}(\Psi_{f}) =  Q_{N,0} \circ \tilde{B}(\a) $ is smooth.
\end{proof}

However, in higher dimensions, $B$ may not be smooth, and similarly, neural networks such as ReLU networks may also not be smooth. But this is not a crucial issue because these functions are locally Lipschitz, as we can consider the subgradient descent method, Lemma \ref{lemma:stochastic}.

\begin{remark}
\label{rmk:endpoint}
If $n > 1$, then $B : \mathbb{R}_{> 0}^N \rightarrow \mathsf{Bar}, \a \mapsto \mathsf{bcd}_{n}^{\mathcal{R},\ell}(\Psi_{f})$ is continuous but may not be smooth. For example, consider a ray $\ell(t) = \sqrt{N} t \cdot {\a \over \lVert \a \rVert}$. Since the minimum function is not smooth, $\mathsf{bcd}_{2}^{\mathcal{R},\ell}(\Psi_{f})=\Biggl\{ \left(0, \min	\left({r_{L_{i_1}}^f \over \sqrt{N}a_{i_1} / \lVert \mathbf{a} \rVert}, {r_{L_{i_2}}^f \over \sqrt{N}a_{i_2} / \lVert \mathbf{a} \rVert} \right) \sqrt{3}  \right]: 1 \le i_1 < i_2 \le N \Biggr\}$ is not smooth.
\end{remark}

Now, we introduce one of the main theorems, the existence theorem of a ray that minimizes the loss function $\mathcal{L}$ defined in Table \ref{table:definition}. The proof of the theorem, based on Lemma \ref{lemma:B} and Lemma \ref{lemma:V} in Section \ref{sec:convergence}, can show that $\mathcal{L} = V \circ B = \tilde{V} \circ \tilde{B}$ is locally Lipschitz.

\begin{theorem}[Variational principle of EMPH]
\label{thm:variation principle}
For a compact subset $\mathcal{C} \subset \mathbb{R}_{> 0}^N$, the function $\mathcal{L}|_{\mathcal{C}} : \mathcal{C} \rightarrow \mathbb{R}$ is continuous, ensuring that it has a minimum value.
\end{theorem}

This theorem signifies that there is an optimal ray, but, \eqref{eq:objective} is not convex in general. Thus the uniqueness of the optimal ray is not guaranteed. Following this theorem, the next step is to demonstrate how we can find such a ray.

\subsubsection{Convergence analysis}
\label{sec:convergence}
In the optimization problem presented in \eqref{eq:objective}, it is still possible to find a feasible solver even if $\mathcal{L}$ is neither smooth nor convex which includes the case mentioned in Remark \ref{rmk:endpoint}, and the case that the activation function is non-smooth. Therefore, it is necessary to use generalized gradient descent to minimize the nonsmooth, nonconvex loss function. The Clarke subgradient serves as a generalization of the gradient for nonsmooth and nonconvex functions. According to Rademacher's theorem, a locally Lipschitz function $f:\R^n \rightarrow \R$ has a Clarke subgradient. If $f$ is smooth, then the Clarke subgradient is the gradient of $f$. If $f$ is a convex function, then it is the subgradient of $f$ \cite{clarke1975generalized}. The following lemma guarantees the convergence of the projected stochastic subgradient descent method using the Clarke subgradient. From this point, we utilize the properties of definable functions to prove Theorem \ref{thm:convergence}. In an o-minimal structure, a definable function, roughly speaking, satisfies the properties of a piecewise smooth function, and a semi-algebraic function is a special case of a definable function. The basic concepts of o-minimal structures and definable functions are summarized in Appendix \ref{appendix:o-minimal}.

\begin{lemma}[Corollary 6.4, \cite{davis2020stochastic}]
\label{lemma:stochastic}
Let $f:\mathbb{R}^d \rightarrow \mathbb{R}$ be both a locally Lipschitz function and a definable function. Consider the projected stochastic subgradient descent method
\begin{equation*}
x_{k+{1 \over 2}} = x_k - \alpha_k \zeta (x_k, \omega_k) \quad \text{(Gradient step)}, 
\end{equation*}
\begin{equation*}
x_{k+1} = \mathsf{proj}_{\mathcal{X}} \left(x_{k+{1 \over 2}}\right) \quad \text{(Projection step)}, 
\end{equation*}
where $\left\{ \zeta (x_k, \omega_k) \right\}_k$ is a stochastic estimator for the Clarke subgradient of $f$ at $x_k$, $\left\{ \alpha_k \right\}_{k \ge 1}$ is a sequence of step-sizes, $\mathcal{X}$ is a closed and definable constraint region, and $\mathsf{proj}_{\mathcal{X}}: \R^d \rightarrow \mathcal{X}$ is a projection map. If the following conditions are satisfied, then $x_k$ almost surely converges to Clarke stationary point:
\begin{enumerate}
    \item (Proper step-size) $\alpha_k \ge 0$, $\sum\limits_{k=1}^{\infty} \alpha_k = \infty$ and $\sum\limits_{k=1}^{\infty} \alpha_k^2 < \infty$.
    \item (Boundedness) $\sup\limits_{k \ge 1} \lVert x_k \rVert < \infty$
    \item (Noise constraint) There exists a function $p : \R^d \rightarrow \R_{> 0}$, that is bounded on bounded sets such that the expectation satisfies that  $\mathbb{E}_{\omega}	\left[\zeta(x,\omega)\right] \in \partial f (x), \mathbb{E}_{\omega}  \left[ \lVert \zeta(x_k, \omega_k) \rVert^2 \right] \le p(x)$ for all $x \in \mathcal{X}$ and $\mathbb{E}_{\omega} 	\left[\sup\limits_{k \ge 1} \lVert \zeta(z_k,\omega) \rVert \right] < \infty$ for every convergent sequence $	\left\{ z_k \right\}_{k \ge 1}$. 
\end{enumerate}
\end{lemma}

\vspace{0.5cm}
To apply Lemma \ref{lemma:stochastic}, we need to prove that $\mathcal{L}$ is both locally Lipschitz and definable. Our strategy involves verifying that both $\tilde{B}$ and $\tilde{V}$ are locally Lipschitz and definable functions, which leads to $\mathcal{L}$ being locally Lipschitz and definable. In order to prove the main theorem of this section, Theorem \ref{thm:convergence}, we use Lemmas \ref{lemma:B} and \ref{lemma:V}. Lemma \ref{lemma:B} demonstrates that $\tilde{B}$ is locally Lipschitz and definable. Lemma \ref{lemma:V} shows that $\tilde{V}$ is locllay Lipschitz and definable. The properties of definable function are detailed in Appendix \ref{appendix:o-minimal}, and those minor propositions regarding locally Lipschitz functions along with their proofs are given in Appendix \ref{appendix:propositions}.

\begin{lemma} 
\label{lemma:B}
Let $\tilde{B} : \R^N_{>0} \rightarrow \R^{2p} \times \R^{q}$ be a map that satisfies $B = Q_{p,q} \circ \tilde{B}$. Then $\tilde{B}$ is locally Lipschitz and definable.
\end{lemma}

\begin{proof}
Let 
\begin{align}
b_{L,n_L}(\a) = \begin{cases}
{2r_L^f \sin\left(\pi {k \over 2k+1}\right)  \over \sqrt{N}a_L / \lVert \mathbf{a} \rVert}, & \mbox{if }n_L=2k+1 \\
0, & \mbox{otherwise }
\end{cases}, 
\label{eq:b}
\\
d_{L,n_L}(\a) = \begin{cases}
{2r_L^f \sin\left(\pi {k+1 \over 2k+3}\right)  \over \sqrt{N}a_L / \lVert \mathbf{a} \rVert}, & \mbox{if }n_L=2k+1 \\
\infty, & \mbox{otherwise }
\end{cases}.
\label{eq:d}
\end{align}

Note that for any $(b,d] \in \mathsf{bcd}_{n}^{\mathcal{R},\ell}(\Psi_{f})$, $b$ is expressed by the maximum of $\left\{b_{L,n_L} \right\}_{L}$ and $d$ is expressed by the minimum of $\left\{d_{L,n_L}\right\}_{L}$ in Theorem \ref{thm:EMPH}. Since $b_{L,n_L}$ and $d_{L,n_L}$ are smooth functions on a compact subset of $\R_{> 0}^N$, they are locally Lipschitz by Proposition \ref{prop:locally Lipschitz}. By Proposition \ref{prop:locally Lipschitz2}, both $b$ and $d$ are locally Lipschitz. Similarly, the maps $b_{L,n_L}$ and $d_{L,n_L}$ are semialgebraic by Proposition \ref{prop:semi properties}, and their maximum and minimum are also semialgebraic by Proposition \ref{prop:min semi}. Therefore, by Proposition \ref{prop:semi-definable}, $\tilde{B}$ is definable.
\end{proof}

\begin{lemma}
\label{lemma:V}
$\tilde{V}= l(y, \cdot) \circ \phi \circ NN \circ  \Phi \circ \I \circ Q_{p,q}$ is locally Lipschitz and definable.

\end{lemma}

\begin{proof}
\begin{enumerate}
\item $l(y, \cdot) \circ \phi$ is smooth function on a compact set. By Proposition \ref{prop:locally Lipschitz}, it is locally Lipschitz.
\item $NN$ is locally Lipschitz since it consists of affine transformations and ReLU functions.
\item $\I$ is locally Lipschitz, as established by Proposition \ref{prop:persistence image}.
\item $Q_{p,q}$ is locally Lipschitz by Proposition 3.2 in \cite{leygonie2021framework}.
\item $\I \circ Q_{p,q}, NN, \phi$ and $l(y,\cdot)$ are definable functions: \\
Since the sum and product of definable functions are definable and the exponential function is definable, $\I \circ Q_{p,q} \left( (b_1,d_1), \cdots, (b_p, d_p) \right) = {1 \over 2\pi\sigma^2} \sum\limits_{i=1}^p (d_i - b_i) e^{-{{\left[ (x-b_i)^2 + (y - (d_i - b_i))^2 \right]} \over 2\sigma^2}}$ is also definable. According to Corollary 5.11 in \cite{davis2020stochastic}, ReLU, softmax, and cross entropy functions are definable.
\item Claerly, $\Phi$ is locally Lipschitz and definable.
\end{enumerate}
Since every component of $\tilde{V}$ consists of locally Lipschitz functions and definable functions, $\tilde{V}$ is locally Lipschitz and definable. \end{proof}

For simplicity, define $\boldsymbol{\rho}$ as $\boldsymbol{\rho}(\a) = {\a \over \lVert \a \rVert}$. Now, consider the following update rule:
\begin{equation}
\begin{aligned}
W^{[l]}_{k+1} &= W^{[l]}_{k} - \alpha_k {\partial \mathcal{L} \over \partial {W^{[l]}_{k}}}, \\
b^{[l]}_{k+1} &= b^{[l]}_{k} - \alpha_k 
{\partial \mathcal{L} \over \partial {b^{[l]}_{k}}}, \\
\a_{k+{1 \over 2}} &= \a_k - \alpha_k
{\partial \mathcal{L} \over \partial {\a_{k}}}, \\
\a_{k+1} &= \mathsf{proj}_{\R^N_{m,M}} \left(\boldsymbol{\rho}\left(\a_{k+{1 \over 2}}\right)\right).
\end{aligned}
\label{eq:update rule}
\end{equation}

As mentioned earlier, the problem \eqref{eq:objective} is inherently nonconvex and nonsmooth. Thus our aim is to find a stationary point within the loss function. Note that Clarke stationary points are not necessarily local minima \cite{li2020understanding}.

Let $\mathbf{W}_k := \left( W^{[1]}_{k}, \cdots, W^{[h]}_{k}\right)$ and $\mathbf{b}_k := \left( b^{[1]}_{k}, \cdots, b^{[h]}_{k}\right)$. We solve the problem \eqref{eq:objective} based on the following theorem.

\begin{theorem}
\label{thm:convergence}
The sequences $\mathbf{W}_k, \mathbf{b}_k$ and $\a_k$ in \eqref{eq:update rule} converge to Clarke stationary points of $\mathcal{L}$.
\end{theorem}

\begin{proof}
Note that $V \circ B = V \circ Q_{p,q} \circ \tilde{B} = \tilde{V} \circ \tilde{B}$. By Lemmas \ref{lemma:B}, \ref{lemma:V} and \ref{lemma:stochastic}, this statement is true.
\end{proof}

\subsubsection{Exact formula of gradient}
\label{sec:Exact formula of gradient}
In this section, we discuss a method that updates efficiently the filtration ray in the backward process. Figure \ref{fig:Pipeline_of_optimization2} illustrates that the exact barcode formula allows efficient computation of the forward process. But the computation of the barcode involves multiple intermediate steps, which makes the procedure of updating the direction through automatic differentiation inefficient. By using the exact gradient of the loss function with respect to the direction vector, we can update the direction vector directly without relying on automatic differentiation, thereby allowing efficient computation of the backward process (e.g. see Example \ref{Ex:time}).

\begin{lemma}
\label{lemma:derivative_rho}
The derivative of $\boldsymbol{\rho}(\a)={\a \over \lVert \a \rVert} = (\rho_1, \cdots, \rho_N)$ is given by the following:
$${\partial \boldsymbol{\rho} \over \partial \a} = {1 \over \lVert \a \rVert} \mathbf{1} - {1 \over \lVert \a \rVert^3} \a^T \a \in \mathbb{R}^{N \times N}.$$
Here, we regard $\a \in \mathbb{R}^{1 \times N}$ as a row vector and $\mathbf{1}$ is an identity matrix.
\end{lemma}




\begin{theorem}[Exact formula of gradients]
\label{thm:gradient}
The explicit forms of the gradients in \eqref{eq:update rule} are as follows:
\begin{equation}
\begin{aligned}
{\partial \mathcal{L} \over \partial W^{[l]}} &= \left(\delta^{[l]}\right)^T \left(g^{[l-1]}(z_{l-1})\right)^T, \\
{\partial \mathcal{L} \over \partial b^{[l]}} &= \left(\delta^{[l]}\right)^T,\\
{\partial \mathcal{L} \over \partial \a} &= \left(\delta^{[1]} W^{[1]}  \right)  \left( {1 \over \I_{\max} - \I_{\min}} \right) 
\left({\partial \I \over \partial \mathsf{b}}{\partial \mathsf{b} \over \partial \boldsymbol{\rho}} + {\partial \I \over \partial \mathsf{d}}{\partial \mathsf{d} \over \partial \boldsymbol{\rho}} \right)
 \left( {1 \over \lVert \mathbf{a} \rVert} \mathbf{1} - {1 \over \lVert \mathbf{a} \rVert^3} \mathbf{a}^T \mathbf{a} \right),   
\end{aligned}
\label{eq:gradients}
\end{equation}
where $z_l = W^{[l]} g^{[l-1]}(z_{l-1}) + b^{[l]}, \delta^{[l]} = {\partial \mathcal{L} \over \partial z_{l}}$ and the persistence image $ z_0 = \I = (I_1, \cdots, I_{r^2}),$ with $\I_{max} = \max (I_1, \cdots, I_{r^2}),  \I_{min} = \min (I_1, \cdots, I_{r^2})$. Additionally, $\mathsf{b} = (b_1, \cdots, b_p)$ and $\mathsf{d} = (d_1, \cdots, d_p)$ represent the coordinates of birth and death times as defined in Definition \ref{def:smooth}, and $\boldsymbol{\rho}(\a) = {\a \over \lVert \a \rVert} = (\rho_1, \cdots, \rho_N)$.

Each component in ${\partial \mathcal{L} \over \partial \a}$ is
\begin{align}
{\partial I_{v} \over \partial b_j} =& -\sum\limits_{(b_i,d_i] \in \mathcal{B}} g_{b_i, d_i} (x_v,y_v) \nonumber\\
+&(d_j - b_j) \cdot {d_j - 2b_j \over \sigma^2} \sum\limits_{(b_i,d_i] \in \mathcal{B}} g_{b_i, d_i} (x_v,y_v) \label{eq:derivative I1} \\
+&(d_j - b_j) \cdot {1 \over \sigma^2} \sum\limits_{(b_i,d_i] \in \mathcal{B}} (x_v-y_v) g_{b_i, d_i} (x_v,y_v), \nonumber \\
{\partial I_{v} \over \partial d_j} =& \sum\limits_{(b_i,d_i] \in \mathcal{B}} g_{b_i, d_i} (x_v,y_v) \nonumber\\
-&(d_j - b_j) \cdot {d_j - b_j \over \sigma^2}  \sum\limits_{(b_i,d_i] \in \mathcal{B}} g_{b_i, d_i} (x_v,y_v) \label{eq:derivative I2} \\
+&(d_j - b_j) \cdot {1 \over \sigma^2} \sum\limits_{(b_i,d_i] \in \mathcal{B}} y_v g_{b_i, d_i} (x_v,y_v), \nonumber
\end{align} 
and Clarke subgradients
\begin{align}
{\partial b_j \over \partial \rho_L} \in& \ {\partial \over \partial \rho_L}	\left( \max \left\{ b_{L, n_L}(\boldsymbol{\rho}): \sum\limits_{L=1}^{N} n_L = n \right\}\right), \label{eq:derivative b}\\
{\partial d_j \over \partial \rho_L} \in& \ {\partial \over \partial \rho_L} 	\left( \min \left\{ d_{L, n_L}(\boldsymbol{\rho}):  \sum\limits_{L=1}^{N} n_L = n \right\}\right), \label{eq:derivative d}
\end{align}    
where $(b_j, d_j] = J_{1}^{n_1, \ell}\bigcap \cdots \bigcap J_{N}^{n_N, \ell} \in \mathsf{bcd}_{n}^{\mathcal{R},\ell}(\Psi_{f})$ in \eqref{eq:EMPH}, and $b_{L, n_L}$ and $d_{L, n_L}$ are defined in \eqref{eq:b} and \eqref{eq:d}, respectively. In particular, if we consider a one-dimensional barcode $(n=1)$, $j$ can be indexed such that $(b_j, d_j] = \Bigl(0, \sqrt{{3 \over N}}{ r_j^f   \over \rho_j } \Bigr]$ for $1 \le j \le N$. Then, we get
\begin{align}
{\partial b_j \over \partial \rho_L} =& \ 0 \quad \text{ 
for } 1\le j,L \le N, \label{eq:derivative b,1}\\
{\partial d_j \over \partial \rho_L} =& \begin{cases}
-\sqrt{{3 \over N}}{r_L^f \over \rho_L^2}, & \mbox{if } j =L \\
0, & \mbox{otherwise}
\end{cases}. \label{eq:derivative d,1}   
\end{align}

\end{theorem}
\begin{proof}
The derivatives ${\partial \mathcal{L} \over \partial W^{[l]}}$ and ${\partial \mathcal{L} \over \partial b^{[l]}}$ are well-known gradients in deep neural networks \cite{clark2017computing},
\begin{align*}
{\partial \mathcal{L} \over \partial \a} &= {\partial \mathcal{L} \over \partial \Phi} {\partial \Phi \over \partial \I} 	\left({\partial \I \over \partial \mathsf{b}}{\partial \mathsf{b} \over \partial \boldsymbol{\rho}} + {\partial \I \over \partial \mathsf{d}}{\partial \mathsf{d} \over \partial \boldsymbol{\rho}} \right){\partial \boldsymbol{\rho} \over \partial \a} 
\\ 
&= \left(\delta^{[1]} W^{[1]}\right)  \left( {1 \over \I_{\max} - \I_{\min}} \right) 
\left({\partial \I \over \partial \mathsf{b}}{\partial \mathsf{b} \over \partial \boldsymbol{\rho}} + {\partial \I \over \partial \mathsf{d}}{\partial \mathsf{d} \over \partial \boldsymbol{\rho}} \right)
 \left( {1 \over \lVert \a \rVert} \mathbf{1} - {1 \over \lVert \a \rVert^3} \a^T \a \right)   
\end{align*}

Equations \eqref{eq:derivative I1} and \eqref{eq:derivative I2} can be directly computed from the definition of $\I$ in \eqref{eq:persistence image}. Equations \eqref{eq:derivative b} and \eqref{eq:derivative d} are induced from $(b_j, d_j] = J_{1}^{n_1, \ell}\bigcap \cdots \bigcap J_{N}^{n_N, \ell} =  \biggl( \max \Bigl\{ b_{L, n_L}(\boldsymbol{\rho}): \sum\limits_{L=1}^N n_L= n \Bigr\}, \min \Bigl\{ d_{L, n_L}(\boldsymbol{\rho}): \sum\limits_{L=1}^N n_L= n \Bigr\} \biggr]$. For \eqref{eq:derivative b,1} and \eqref{eq:derivative d,1}, set $b_j = \max \left\{ b_{j, n_j}(\boldsymbol{\rho}^1, \cdots, \boldsymbol{\rho}^R): n_j =1 \text{ and } n_1, \cdots, n_{j-1}, n_{j+1}, \cdots, n_N = 0 \right\} = 0$ and $d_j = \min \bigl\{ d_{j, n_j}(\boldsymbol{\rho}^1, \cdots, \boldsymbol{\rho}^R): n_j =1 \text{ and } n_1, \cdots, n_{j-1}, n_{j+1}, \cdots, n_N = 0 \bigr\} = \sqrt{{3 \over N}}{ r_j^f   \over \rho_j }$. Differentiating these with respect to $\rho_L$ yields \eqref{eq:derivative b,1} and \eqref{eq:derivative d,1}.
\end{proof}

Using these formulas, we can directly update $\a$ in \eqref{eq:update rule} and calculate efficiently the backward process. In Example \ref{Ex:time}, we will show the advantages of the exact computation by comparing the computational costs.


\subsection{Filtration curve learning}
\label{subsec:curved}
The optimal filtration ray within a multi-parameter space can be generalized to curved filtration that offers a more accurate reflection of the characteristics of the data.
\begin{definition}[Filtration curve and persistent homology] 
For an increasing curve filtration $\c = (c_1,\cdots , c_N) : [0,Q] \rightarrow \R^N$ (with $\c' > 0$), we define a one-parameter filtration $\mathcal{R}^{\c}\left(\Psi_{f}\right) : \mathbb{R} \rightarrow \mathbf{Simp}$ by setting $\mathcal{R}^{\c}_{t} \left(\Psi_{f}\right) = \mathcal{R}_{c_1 (t)}\left(r_1^f \cdot \mathbb{S}^1\right) \times \cdots \times \mathcal{R}_{c_N (t)}\left(r_N^f \cdot \mathbb{S}^1\right)$. We denote the barcode induced from this filtration as $\mathsf{bcd}_{*}^{\mathcal{R},\c}(\Psi_{f})$. 
\end{definition}
As in the proof of Theorem \ref{thm:EMPH}, we can easily deduce the exact formula for the barcode $\mathsf{bcd}_{*}^{\mathcal{R},\c}(\Psi_{f})$.   

\begin{lemma}
Let an increasing curve $\c : [0,Q] \rightarrow \R^N$ with $\c(0) = \mathbf{0}$ have a constant speed of $\sqrt{N}$. Then the barcode is given by the following:  
\begin{equation}
\mathsf{bcd}_{n}^{\mathcal{R},\c}(\Psi_{f}) = \left\{ J_{1}^{n_1, \c}\bigcap \cdots \bigcap J_{N}^{n_N, \c} :  J_{L}^{n_L, \c} \in  \mathsf{bcd}_{n_L}^{\mathcal{R},\c} 	\left( r_L^f \cdot \mathbb{S}^1 \right)  \ \text{and} \ \sum\limits_{L=1}^{N} n_L = n \right\},
\label{eq:curve barcode}
\end{equation}

where 
\begin{equation*}
J_{L}^{n_L, \c} = \begin{cases} \left(0,\infty\right), & \mbox{if }n_L=0 \\ \left(c_L^{-1}\left(2r_L^f \sin\left(\pi {k \over 2k+1}\right) \right)  , c_L^{-1}\left(2r_L^f \sin\left(\pi {k+1 \over 2k+3}\right)  \right)\right], & \mbox{if }n_L=2k+1 \\ \ \emptyset, & \mbox{otherwise}
\end{cases}.
\end{equation*}
\end{lemma}

\begin{proof}
Note that
$$\mathsf{bcd}_{n}^{\mathcal{R},\c}(r_L^f \cdot \mathbb{S}^1)=
\begin{cases}
\left\{ \left(0,\infty\right) \right\}, & \mbox{if }n=0 \\
\left\{ J_{L}^{2k+1, \c} \right\}, & \mbox{if }n=2k+1 \\
\emptyset, &\mbox{ otherwise}
\end{cases},$$
where 
$$J_{L}^{2k+1, \c} := \left( \min\limits_{t} c_L(t) \ge 2r_L^f \sin\left(\pi{k \over 2k +1}\right) , \min\limits_{t}c_L(t) \ge 2r_L^f \sin\left(\pi {k+1 \over 2k+3}\right) \right] 
$$

$$= \left(c_L^{-1}\left(2r_L^f \sin\left(\pi {k \over 2k+1}\right) \right)  , c_L^{-1}\left(2r_L^f \sin\left(\pi {k+1 \over 2k+3}\right)  \right)\right].$$

By applying the persistent K\"unneth formula (\cite{gakhar2019k, kim2022exact}) to the metric space $(r_1^f \cdot \mathbb{S}^1 \times \cdots \times r_N^f \cdot \mathbb{S}^1)$, with maximum metric, we obtain \eqref{eq:curve barcode}.
\end{proof}
We can easily verify that the barcode formula for filtration ray is a special case of filtration curve.
\begin{example}
Let $c_L(t) = \sqrt{N}t \cdot {a_L \over \lVert \a \rVert}$. Then we have its inverse $c_L^{-1}(t) = {t \over \sqrt{N}a_L / \lVert \mathbf{a} \rVert}$ and $\mathsf{bcd}_{n}^{\mathcal{R},\c}(\Psi_{f}) = \mathsf{bcd}_{n}^{\mathcal{R},\ell}(\Psi_{f})$ as in Theorem \ref{thm:EMPH}.      
\end{example}

Extending the filtration ray optimization directly to the filtration curves is not straightforward, due to the fundamental difference between them. The classification with the filtration ray is parameterized by a single direction vector, but the filtration curve is generally continuously parameterized, not finitely unlike the filtration ray. This means that the filtration curve learning involves infinitely many parameter updates. We address this issus by approximating the curve as a piecewise linear function, which allows the filtration curve to be finitely parameterized.

Given a curve $\c = (c_1, \cdots, c_N) : [0,Q] \rightarrow \R^N$ with $\c(0) = \mathbf{0}$, consider a piecewise line approximation $\tilde{\c} = (\tilde{c}_1, \cdots, \tilde{c}_N) : [0,Q] \rightarrow \R^N$, where
\begin{equation}
\label{eq:piecewise}
\tilde{c}_L(t) = \sqrt{N}{Q \over R}\sum\limits_{i=1}^{s-1}{a_L^i \over \lVert \a^i \rVert} + \sqrt{N}\left(t-{(s-1)Q \over R}\right){a_L^s \over \lVert \a^s \rVert}, \quad (s-1){Q \over R} \le t \le s{Q \over R},    
\end{equation}
$\a^s = (a^s_1, \cdots, a^s_N) > 0$ and $\c(s{Q \over R}) = \tilde{\c}(s{Q \over R})$ for $s=1,\cdots,R$. Here, $\a^s$ is the direction vector of the $s$th line segment and $R$ represents the number of line segments.

The following theorem states that the barcode is parameterized by the finite number of parameters $\a^1, \cdots, \a^s$ in an explicit form.

\begin{theorem}
\label{thm:piecewise barcode}
For a piecewise linear approximation of a curve $\tilde{\c}$, we have

\begin{equation}
\mathsf{bcd}_{n}^{\mathcal{R},\tilde{\c}}(\Psi_{f}) = \left\{ J_{1}^{n_1, \tilde{\c}}\bigcap \cdots \bigcap J_{N}^{n_N, \tilde{\c}} :  \sum\limits_{L=1}^{N} n_L = n \right\},
\label{eq:piecewise barcode}
\end{equation}
where
{\tiny
\begin{equation*}
J_{L}^{n_L, \tilde{\c}} = \begin{cases} \left(0,\infty\right), & \mbox{if }n_L=0 \\ \left({2r_L^f \sin\left(\pi {k \over 2k+1}\right)  \over \sqrt{N} a_L^{s_{b,L}} / \lVert \a^{s_{b,L}} \rVert} - {Q \over R}{\sum\limits_{i=1}^{{s_{b,L}}} {a_L^i \over \lVert \a^i \rVert} \over a_L^{s_{b,L}} / \lVert \a^{s_{b,L}} \rVert} + {s_{b,L}}{Q \over R}, {2r_L^f \sin\left(\pi {k +1 \over 2k+3}\right)  \over \sqrt{N} a_L^{s_{d,L}} / \lVert \a^{s_{d,L}} \rVert} - {Q \over R}{\sum\limits_{i=1}^{{s_{d,L}}} {a_L^i \over \lVert \a^i \rVert} \over a_L^{s_{d,L}} / \lVert \a^{s_{d,L}} \rVert} + {s_{d,L}}{Q \over R} \right], & \mbox{if }n_L=2k+1 \\ \ \emptyset, & \mbox{otherwise}
\end{cases},
\end{equation*}
}
$s_{b,L}$ satisfies $\tilde{c}_L \left( (s_{b,L}-1) {Q \over R}\right) \le 2r_L^f \sin\left(\pi {k \over 2k+1}\right) \le \tilde{c}_L \left( s_{b,L} {Q \over R}\right)$ and $s_{d,L}$ satisfies $\tilde{c}_L \left( (s_{d,L}-1) {Q \over R}\right) \le 2r_L^f \sin\left(\pi {k+1 \over 2k+3}\right) \le \tilde{c}_L \left( s_{d,L} {Q \over R}\right)$.
\end{theorem}

\begin{proof}
Since $$\tilde{c}_L^{-1}(x) = {x - \sqrt{N}{Q \over R}\sum\limits_{i=1}^{s-1}{a_L^i \over \lVert \a^i \rVert} + \sqrt{N}{(s-1) \over R}Q {a_L^s \over \lVert \a^s \rVert} \over \sqrt{N}a_L^s / \lVert \a^s \rVert}, \quad \tilde{c}_L\left( (s-1){Q \over R} \right)\le x \le \tilde{c}_L\left( s{Q \over R} \right),$$ the birth time is 
\begin{align}
\tilde{c}_L^{-1}\left(2r_L^f \sin\left(\pi {k \over 2k+1}\right) \right) &= {2r_L^f \sin\left(\pi {k \over 2k+1}\right) - \sqrt{N}{Q \over R}\sum\limits_{i=1}^{s_{b,L}-1}{a_L^i \over \lVert \a^i \rVert} + \sqrt{N}{(s_{b,L}-1) \over R}Q {a_L^{s_{b,L}} \over \lVert \a^{s_{b,L}} \rVert} \over \sqrt{N}a_L^{s_{b,L}} / \lVert \a^{s_{b,L}} \rVert}  \nonumber \\ &= {2r_L^f \sin\left(\pi {k \over 2k+1}\right)  \over \sqrt{N} a_L^{s_{b,L}} / \lVert \a^{s_{b,L}} \rVert} - {Q \over R}{\sum\limits_{i=1}^{{s_{b,L}}} {a_L^i \over \lVert \a^i \rVert} \over a_L^{s_{b,L}} / \lVert \a^{s_{b,L}} \rVert} + {s_{b,L}}{Q \over R}.
\label{eq:converge}
\end{align}
The death time can be shown similarly.   
\end{proof}

\begin{corollary}
The approximated barcode $\mathsf{bcd}_{n}^{\mathcal{R},\tilde{\c}}(\Psi_{f})$ converges to $\mathsf{bcd}_{n}^{\mathcal{R},\c}(\Psi_{f})$ as $R \rightarrow \infty$.      
\end{corollary}

\begin{proof}
In \eqref{eq:converge},
\begin{align*}
{2r_L^f \sin\left(\pi {k \over 2k+1}\right)  \over \sqrt{N} a_L^{s_{b,L}} / \lVert \a^{s_{b,L}} \rVert} - {Q \over R}{\sum\limits_{i=1}^{{s_{b,L}}} {a_L^i \over \lVert \a^i \rVert} \over a_L^{s_{b,L}} / \lVert \a^{s_{b,L}} \rVert} + {s_{b,L}}{Q \over R} &  \\
= \tilde{c}_L^{-1}\left(2r_L^f \sin\left(\pi {k \over 2k+1}\right) \right) & \longrightarrow \ c_L^{-1}\left(2r_L^f \sin\left(\pi {k \over 2k+1}\right) \right) 
\end{align*}
as $R \rightarrow \infty$, and similarly,
\begin{align*}
{2r_L^f \sin\left(\pi {k+1 \over 2k+3}\right)  \over \sqrt{N} a_L^{s_{d,L}} / \lVert \a^{s_{d,L}} \rVert} - {Q \over R}{\sum\limits_{i=1}^{{s_{d,L}}} {a_L^i \over \lVert \a^i \rVert} \over a_L^{s_{d,L}} / \lVert \a^{s_{d,L}} \rVert} + {s_{d,L}}{Q \over R} &  \\
= \tilde{c}_L^{-1}\left(2r_L^f \sin\left(\pi {k+1 \over 2k+3}\right) \right) & \longrightarrow \ c_L^{-1}\left(2r_L^f \sin\left(\pi {k+1 \over 2k+3}\right) \right) 
\end{align*}
as $R \rightarrow \infty$. Since $J_{L}^{n_L, \tilde{\c}}$ in \eqref{eq:piecewise barcode} converges to $J_{L}^{n_L, \c}$ in \eqref{eq:curve barcode}, $\mathsf{bcd}_{n}^{\mathcal{R},\tilde{\c}}(\Psi_{f})$ converges to $\mathsf{bcd}_{n}^{\mathcal{R},\c}(\Psi_{f})$ as $R \rightarrow \infty$.
\end{proof}
Notations are shared in Table \ref{table:definition} except for $B: \left(\mathbb{R}_{m,M}^N\right)^R \rightarrow \mathsf{Bar}, (\mathbf{a}^1, \cdots, \mathbf{a}^R)  \mapsto \mathsf{bcd}_{n}^{\mathcal{R},\tilde{\mathbf{c}}}(\Psi_{f})$. As in Lemma \ref{lemma:smooth} and Remark \ref{rmk:endpoint}, $B$ and $V$ are continuous. Therefore $\mathcal{L}(\mathbf{W},\mathbf{b}, \a^1, \cdots, \a^R) := V \circ \mathsf{bcd}_{n}^{\mathcal{R},\tilde{\mathbf{c}}}(\Psi_{f}) \rightarrow V \circ \mathsf{bcd}_{n}^{\mathcal{R},\c}(\Psi_{f}) = \mathcal{L}(\mathbf{W},\mathbf{b}, \c)$ as $R \rightarrow \infty$, and the filtration curve learning problem can be approximated as the following constrained optimization problem:
\begin{equation}
\label{eq:objective2}
\argmin_{\substack{\mathbf{W},\mathbf{b}, \\ \a^1, \cdots, \a^R \in \mathbb{R}_{m,M}^N}} \mathcal{L}(\mathbf{W},\mathbf{b}, \a^1, \cdots, \a^R).
\end{equation}

Figure \ref{fig:filtration curve learning} schematically illustrates the filtration curve learning of EMPH. The only difference from the filtration ray learning is the number of learning parameters of direction vectors $\a^s$. Therefore, similarly to Lemma \ref{lemma:B}, it can be shown that $\tilde{B}$ is locally Lipschitz and definable. Additionally, the gradient of the loss function with respect to the direction vectors $\a^s$ can be derived. Consequently, Corollary \ref{cor:filtration curve} can be obtained.

\begin{figure}[h]
    \centering
\includegraphics[width=1 \linewidth]{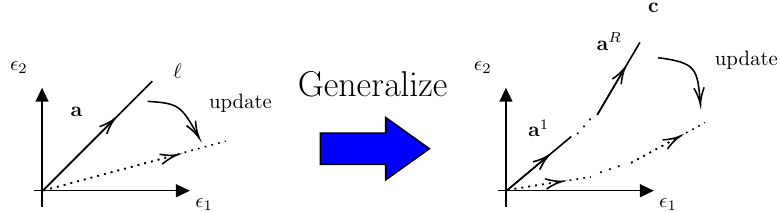}
\caption{Schematic illustration of filtration curve learning in EMPH. Filtration curve learning involves changing the left-hand side figure to the right-hand side figure in Figure \ref{fig:Pipeline_of_optimization2}.}
\label{fig:filtration curve learning}
\end{figure}

\begin{corollary}
\label{cor:filtration curve}
For a piecewise linear approximation of a curve $\c$, as defined in \eqref{eq:piecewise}, the sequences of $\mathbf{W}_k, \mathbf{b}_k$ and $\a^s_k$ in \eqref{eq:objective2} converge to the Clarke stationary points of $\mathcal{L}$ with the following update rule,

\begin{equation*}
\begin{aligned}
W^{[l]}_{k+1} &= W^{[l]}_{k} - \alpha_k {\partial \mathcal{L} \over \partial {W^{[l]}_{k}}}, \\
b^{[l]}_{k+1} &= b^{[l]}_{k} - \alpha_k 
{\partial \mathcal{L} \over \partial {b^{[l]}_{k}}}, \\
\a_{k+{1 \over 2}}^s &= \a_k^s - \alpha_k
{\partial \mathcal{L} \over \partial {\a_{k}^s}}, \\
\a_{k+1}^s &= \mathsf{proj}_{\R^N_{m,M}} \left(\boldsymbol{\rho}\left(\a_{k+{1 \over 2}}^s\right)\right),
\end{aligned}
\label{eq:update rule2}
\end{equation*}
where 

\begin{equation*}
\begin{aligned}
{\partial \mathcal{L} \over \partial W^{[l]}} &= \left(\delta^{[l]}\right)^T \left(g^{[l-1]}(z_{l-1})\right)^T, \\
{\partial \mathcal{L} \over \partial b^{[l]}} &= \left(\delta^{[l]}\right)^T,\\
{\partial \mathcal{L} \over \partial \a^s} &= \left(\delta^{[1]} W^{[1]}\right)  \left( {1 \over \I_{\max} - \I_{\min}} \right) 
\left({\partial \I \over \partial \mathsf{b}}{\partial \mathsf{b} \over \partial \boldsymbol{\rho}^s} + {\partial \I \over \partial \mathsf{d}}{\partial \mathsf{d} \over \partial \boldsymbol{\rho}^s} \right)
 \left( {1 \over \lVert \mathbf{a}^s \rVert} \mathbf{1} - {1 \over \lVert \mathbf{a}^s \rVert^3}  \left(\mathbf{a}^s\right)^T \mathbf{a}^s \right).   
\end{aligned}
\label{eq:gradients2}
\end{equation*}
For the notations above, see Theorem \ref{thm:gradient} with $\boldsymbol{\rho}^s = {\a^s \over \lVert \a^s \rVert} = (\rho_1^s, \cdots, \rho_N^s)$.

Each component of ${\partial \mathcal{L} \over \partial \a^s}$ is as follows: the derivatives ${\partial I_{v} \over \partial b_j}$ and ${\partial I_{v} \over \partial d_j}$ are the same as in \eqref{eq:derivative I1} and \eqref{eq:derivative I2}, respectively.

Clarke subgradients are
\begin{align*}
{\partial b_j \over \partial \rho_L^s} \in& \ {\partial \over \partial \rho_L^s} \left( \max \left\{ b_{L, n_L}(\boldsymbol{\rho}^1, \cdots, \boldsymbol{\rho}^R): \sum\limits_L n_L =n \right\}\right), \\
{\partial d_j \over \partial \rho_L^s} \in& \ {\partial \over \partial \rho_L^s} 	\left( \min \left\{ d_{L, n_L}(\boldsymbol{\rho}^1, \cdots, \boldsymbol{\rho}^R): \sum\limits_L n_L =n \right\}\right),
\end{align*}    
where $(b_j, d_j] = J_{1}^{n_1, \tilde{\c}}\bigcap \cdots \bigcap J_{N}^{n_N, \tilde{\c}} \in \mathsf{bcd}_{n}^{\mathcal{R},\tilde{\c}}(\Psi_{f})$ in \eqref{eq:piecewise barcode} and
\begin{align*}
b_{L,n_L}(\a^1, \cdots, \a^R) = \begin{cases}
{2r_L^f \sin\left(\pi {k \over 2k+1}\right)  \over \sqrt{N} a_L^{s_{b,L}} / \lVert \a^{s_{b,L}} \rVert} - {Q \over R}{\sum\limits_{i=1}^{{s_{b,L}}} {a_L^i \over \lVert \a^i \rVert} \over a_L^{s_{b,L}} / \lVert \a^{s_{b,L}} \rVert} + {s_{b,L}}{Q \over R}, & \mbox{if }n_L=2k+1 \\
0, & \mbox{otherwise }
\end{cases}, 
\\
d_{L,n_L}(\a^1, \cdots, \a^R) = \begin{cases}
{2r_L^f \sin\left(\pi {k +1 \over 2k+3}\right)  \over \sqrt{N} a_L^{s_{d,L}} / \lVert \a^{s_{d,L}} \rVert} - {Q \over R}{\sum\limits_{i=1}^{{s_{d,L}}} {a_L^i \over \lVert \a^i \rVert} \over a_L^{s_{d,L}} / \lVert \a^{s_{d,L}} \rVert} + {s_{d,L}}{Q \over R}, & \mbox{if }n_L=2k+1 \\
\infty, & \mbox{otherwise }
\end{cases}.
\end{align*}
Here, $s_{b,L}$ satisfies $\tilde{c}_L \left( (s_{b,L}-1) {Q \over R}\right) \le 2r_L^f \sin\left(\pi {k \over 2k+1}\right) \le \tilde{c}_L \left( s_{b,L} {Q \over R}\right)$ and $s_{d,L}$ satisfies $\tilde{c}_L \left( (s_{d,L}-1) {Q \over R}\right) \le 2r_L^f \sin\left(\pi {k+1 \over 2k+3}\right) \le \tilde{c}_L \left( s_{d,L} {Q \over R}\right)$. In particular, if we consider a one-dimensional barcode $(n=1)$, $j$ can be indexed such that $(b_j, d_j] = \biggl(0, \sqrt{{3 \over N}}r_j^f {1 \over \rho_j^{s_{d,j}}} - {Q \over R} {\sum\limits_{i=1}^{s_{d,j}} \rho_j^i \over \rho_j^{s_{d,j}}} + s_{d,j} {Q \over R} \biggr]$. Then, we get
\begin{align}
{\partial b_j \over \partial \rho_L^s} =& \ 0 \quad \text{ for } 1\le j,L \le N \text{ and } 1\le s \le R, \label{eq:birth_derivative}\\
{\partial d_j \over \partial \rho_L^s} =& \begin{cases}
-{Q \over R}{1 \over \rho_L^{s_{d,L}}}, & \mbox{if } j =L \text{ and } s<s_{d,L},\\
-\sqrt{{3 \over N}}r_L^f {1 \over 	\left(\rho_L^{s_{d,L}}\right)^2} + {Q \over R} {\sum\limits_{i=1}^{s_{d,L}-1} \rho_L^i \over 	\left(\rho_L^{s_{d,L}}\right)^2}, & \mbox{if } j=L \text{ and } s=s_{d,L}, \\
0, & \mbox{otherwise}
\end{cases}.
\label{eq:death_derivative}
\end{align}
\end{corollary}

\begin{proof}
Here, we only prove \eqref{eq:birth_derivative} and \eqref{eq:death_derivative}. Others can be proved similarly to the proof of Theorem \ref{thm:gradient}. Since $s_{b,L} =0, b_{L,n_L}(\boldsymbol{\rho}^1, \cdots, \boldsymbol{\rho}^R) = 0$ for every $L$ and $\sum\limits_L n_L = 1$. Set $b_j = \max \bigl\{ b_{j, n_j}(\boldsymbol{\rho}^1, \cdots, \boldsymbol{\rho}^R): n_j =1 \text{ and } n_1, \cdots, n_{j-1}, n_{j+1}, \cdots, n_N = 0 \bigr\} = 0$ and $d_j = \min \left\{ d_{j, n_j}(\boldsymbol{\rho}^1, \cdots, \boldsymbol{\rho}^R): n_j =1 \text{ and } n_1, \cdots, n_{j-1}, n_{j+1}, \cdots, n_N = 0 \right\} = \sqrt{{3 \over N}}r_j^f {1 \over \rho_j^{s_{d,j}}} - {Q \over R} {\sum\limits_{i=1}^{s_{d,j}} \rho_j^i \over \rho_j^{s_{d,j}}} + s_{d,j} {Q \over R}$. Differentiating these with respect to $\rho_L^s$ yields \eqref{eq:birth_derivative} and \eqref{eq:death_derivative}.
\end{proof}

\begin{remark}[Comparison with vectorization of multi-parameter persistent homology]
\label{rmk:comparison}
The multi-parameter persistence image and multi-parameter persistence landscape are vectorization techniques derived from the rank invariant of multi-parameter persistent homology, as discussed in \cite{carriere2020multiparameter} and \cite{vipond2020multiparameter}, respectively. Our method shares a common feature with previous studies in that it considers rank invariants (or fibered barcodes). However, there is a distinguishing factor in how we select one-parameter filtrations in the multi-parameter space. Previous research chose one-parameter filtration rays without using machine learning and applied a multi-parameter version of the vectorization method. In contrast, our approach defines a single `optimal' filtration curve (see Figure \ref{fig:comparison}) and utilizes filtration learning to identify this optimal curve. This approach allows for a more concise analysis of the data and facilitates handling $n$-parameter persistence for $n \ge 2$. Furthermore, our method can be extended to find multiple optimized one-parameter filtration curves by generalizing $B$ and it can be combined with the multi-parameter version of the vectorization method. Additionally, we provide an analytical example of the multi-parameter persistence image and landscape in Appendix \ref{app:vectorization}.
\end{remark}

\begin{figure}[h]
\centering
\begin{minipage}{0.4\textwidth}
\centering
\includegraphics[width=\linewidth]{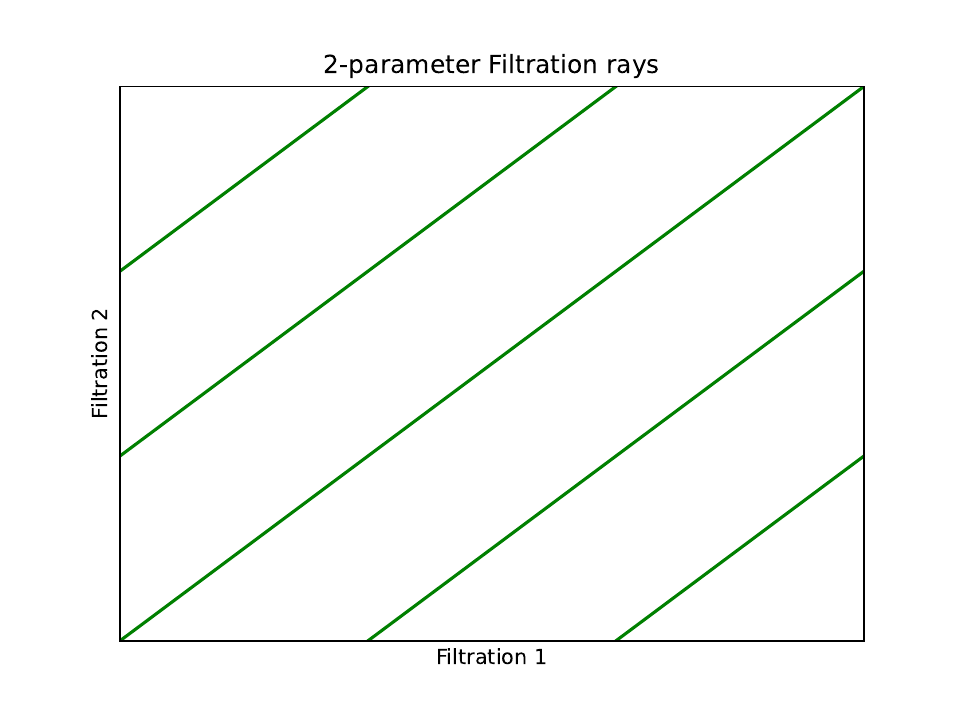}
\end{minipage}
\begin{minipage}{0.1\textwidth}
\centering
$\Longrightarrow$
\end{minipage}
\begin{minipage}{0.4\textwidth}
\centering
\includegraphics[width=\linewidth]{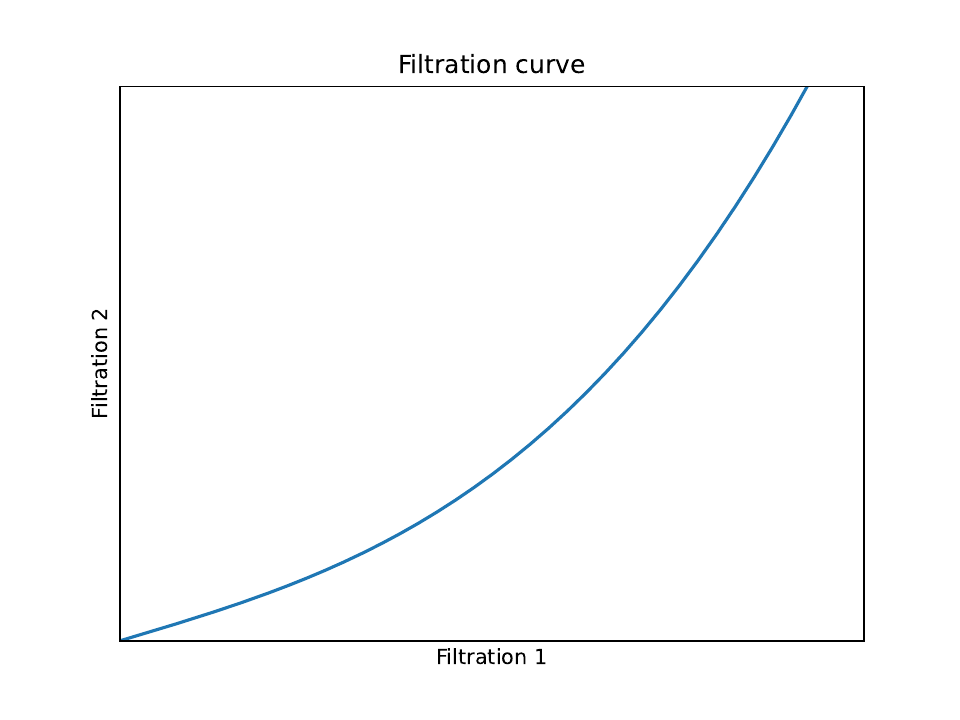}
\end{minipage}
\caption{The multi-parameter persistence image and landscape consider multiple lines in the multi-parameter space (left), while our model pursues finding a single optimal filtration curve (right).}
\label{fig:comparison}
\end{figure}

\begin{section}{Numerical examples}
\label{sec:examples}
In this section, we verify our proposed filtration learning model using time series data. In Examples \ref{Ex:toy1} and \ref{Ex:toy2}, we demonstrate the efficiency of the proposed method of the filtration ray learning and filtration curve learning using synthetic data. In Example \ref{example:ucr}, we compare the performances of our model filtration curve learning and other methods. In Example \ref{Ex:time}, we compare the computational efficiency of using the gradient formula versus not using the gradient formula in the backward step.

Algorithm \ref{alg:filtration main} summarizes the pipeline for solving the optimization problem in \eqref{eq:objective2}. For these examples, we assume that the time-series data are evenly sampled with the length of $\mathfrak{n}$.

\begin{algorithm}[H]
\begin{algorithmic}[1]
\caption{Filtration learning with EMPH (FL-EMPH) }
\label{alg:filtration main}

\STATE \textbf{Input:}  $f_1, \cdots ,f_S : \left\{ {2\pi i \over \mathfrak{n}} : 0 \le i \le \mathfrak{n}-1\right\} \rightarrow \mathbb{R}$ (time-series data) 
\vspace{0.1cm}
\STATE \textbf{Variables:} $E\in \mathbb{Z}_{\ge 0}$ (iterations), $N \in \mathbb{Z}_{\ge 0}$ (degree of truncated Fourier series), $R \in \mathbb{Z}_{\ge 0}$ (number of line segments), $n \in \mathbb{Z}_{\ge 0}$ (dimension of barcode) and $h \in \mathbb{Z}_{\ge 0}$ (number of hidden layers), $M>m>0 $ (maximum and minimum of $\a^s$)

\STATE \textit{Initialize the neural network and direction vectors:} $\mathbf{W}_0 := \left( W^{[1]}_{0}, \cdots, W^{[h]}_{0}\right), \mathbf{b}_0 := \left( b^{[1]}_{0}, \cdots, b^{[h]}_{0}\right)$ and $\a^s_0 $ for $s=1, \cdots, R$.

\FOR{$k = 0,\cdots ,E$}
\FOR{$j = 1,\cdots ,S$}
\STATE Calculate the Fourier transforms $\hat{f}_j$ via fast Fourier transform.

\STATE Calculate $\mathsf{bcd}_{n}^{\mathcal{R},\tilde{\mathbf{c}}}(\Psi_{f_j})$ using Theorem \ref{thm:piecewise barcode}.

\STATE Calculate persistence image of $\mathsf{bcd}_{n}^{\mathcal{R},\tilde{\mathbf{c}}}(\Psi_{f_j})$.

\ENDFOR 

\STATE
Input the persistence image into the neural network. 
\STATE
Update the parameters:
\begin{align*}
W^{[l]}_{k+1} &= W^{[l]}_{k} - \alpha_k {\partial \mathcal{L} \over \partial {W^{[l]}_{k}}}, \\
b^{[l]}_{k+1} &= b^{[l]}_{k} - \alpha_k 
{\partial \mathcal{L} \over \partial {b^{[l]}_{k}}}, \\
\a_{k+{1 \over 2}}^s &= \a_k^s - \alpha_k
{\partial \mathcal{L} \over \partial {\a_{k}^s}}, \\
\a_{k+1}^s &= \mathsf{proj}_{\R^N_{m,M}} \left(\boldsymbol{\rho}\left(\a_{k+{1 \over 2}}^s\right)\right).
\end{align*}

\ENDFOR

\RETURN{} Class prediction
\end{algorithmic}
\end{algorithm}

\begin{example}
\label{Ex:toy1}
Consider the following two time-series, $g_1$ and $g_2$, each with the length of $36$, evenly sampled from $\cos t$ and $\cos 5t$, respectively, in the interval $[0,2\pi]$. Given a time-series $f$, we construct EMPH for the two-parameter filtration, which consists of $\boldsymbol{\epsilon} = (\epsilon_1, \epsilon_5) \mapsto \mathcal{R}_{\epsilon_1}\left(r_1^f \cdot \mathbb{S}^1\right) \times \mathcal{R}_{\epsilon_5}\left(r_5^f \cdot \mathbb{S}^1\right)$. Note that $\boldsymbol{\epsilon}$ is in the frequency domain, which consists of the first and fifth Fourier modes. If we employ sliding window embedding or a diagonal ray for EMPH, $g_1$ and $g_2$ are indistinguishable. In order to distinguish these two, we set up the experiment as follows. We first generate $100$ time-series data for each and used $80\%$ of them for training, with the remainder for testing. We use one hidden layer with the width of 50 for 10,000 epochs. The learning rate of 0.001, the bandwidth of 0.05 for the persistence image, and the resolution of $10^2$ are used for
\begin{equation}
  \left\{  \begin{array}{l}  g^\epsilon_1  = \cos t + 1\cdot \epsilon(0, 1) \\
      g^\epsilon_2 = \cos 5t + 1\cdot \epsilon(0, 1)
      \end{array}
      \right. \nonumber , 
\end{equation}
where $\epsilon(0, 1)$ represents the Gaussian standard normal errors. 

We compare two methods: one that updates the neural network only with the direction vector fixed and the other that utilizes the filtration ray learning where both the neural network and the direction vector are updated. Our experiment shows that the first method achieves accuracy of 52.5\% while the second method achieves perfect accuracy of 100\%. Figure \ref{fig:toy} illustrates how the loss function (left figure) and direction vectors (right figure) behaves at different epochs.

\begin{figure}[h]
  \centering
  \begin{minipage}{0.45\textwidth}
  \centering
    \includegraphics[width=\linewidth]{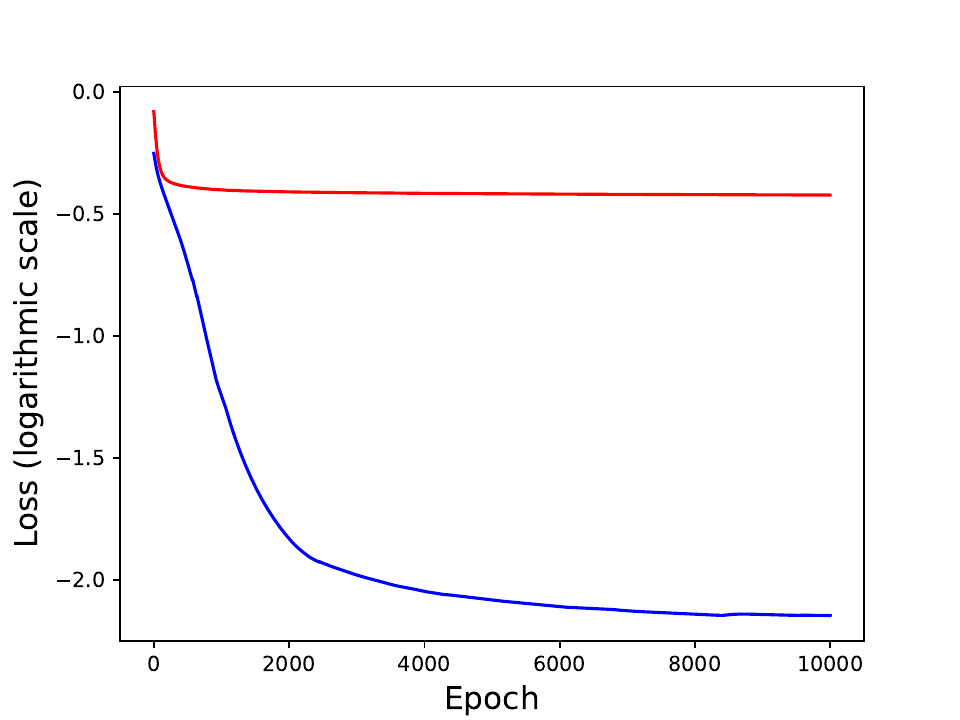}
    \end{minipage}
    \begin{minipage}{0.54\textwidth}
    \includegraphics[width=\linewidth]{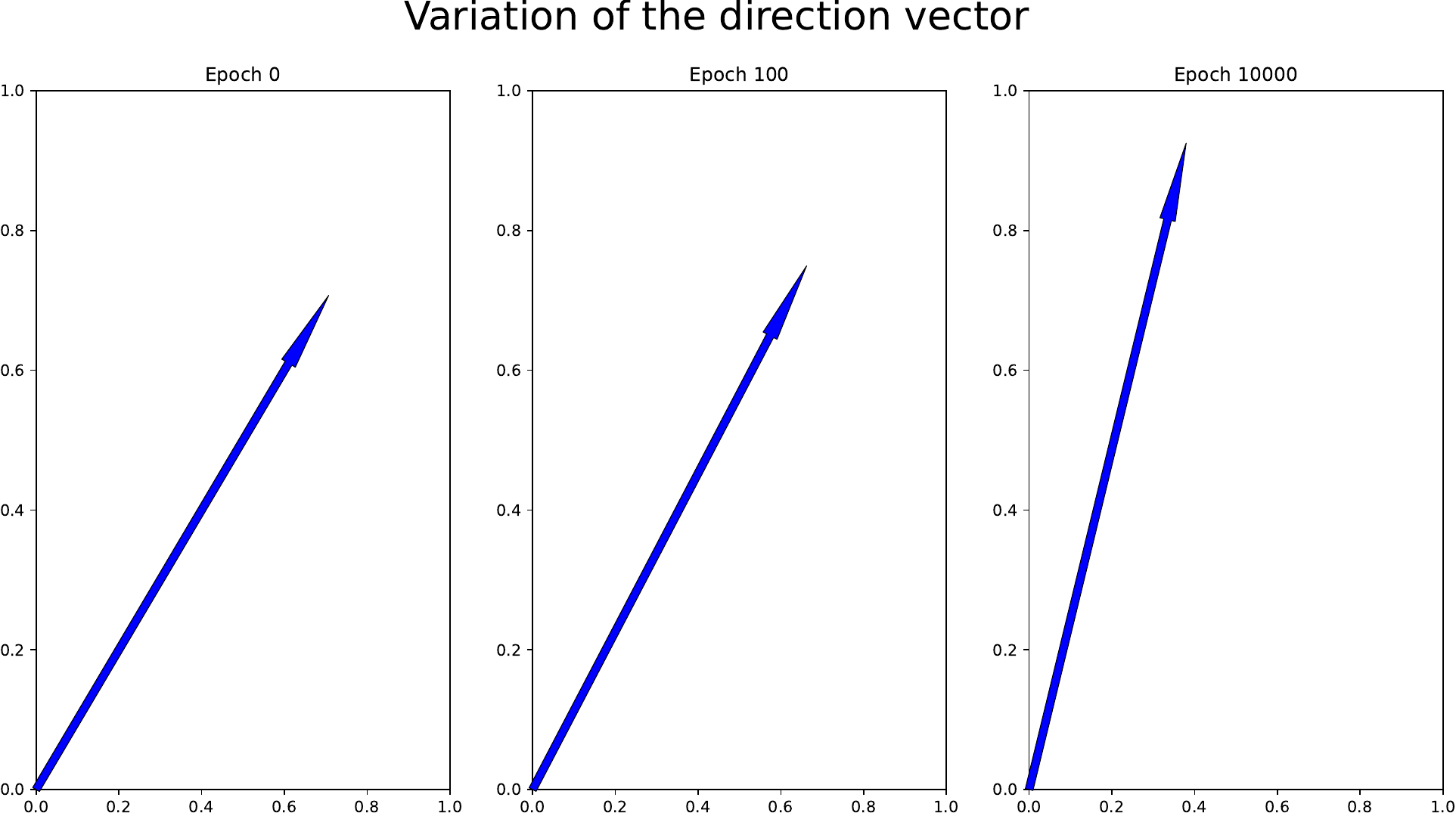}
    \end{minipage}
  \caption{Left: Loss function in logarithmic scale versus epoch number. Red (accuracy: 52.5 \%) solid line represents the case that only the neural network is updated. Blue (accuracy: 100 \%) solid line represents the case that both the neural network and the direction vector are updated. Right: Updated direction vectors at different epochs. The rightmost direction vector is obtained at the last epoch.}
\label{fig:toy}
\end{figure}
\end{example}

\begin{example}
\label{Ex:toy2}
Let $g_{11} = \cos t, g_{12} = \cos 2t, g_2 = 2 \cos t$ and $g_3 = 2 \cos 2t$. Given a time-series $f$, we construct EMPH for the two-parameter filtration, which consists of $\boldsymbol{\epsilon} = (\epsilon_1, \epsilon_2) \mapsto \mathcal{R}_{\epsilon_1}\left(r_1^f \cdot \mathbb{S}^1\right) \times \mathcal{R}_{\epsilon_2}\left(r_2^f \cdot \mathbb{S}^1\right)$. Note that $\boldsymbol{\epsilon}$ is in the frequency domain, which consists of the first and second Fourier modes. To treat $g_{11}$ and $g_{12}$ as identical while distinguishing $g_2$ and $g_3$, simply selecting a single ray is not sufficient. Opting for the diagonal ray enables us to treat $g_{11}$ and $g_{12}$ as the same, yet, it fails to differentiate between $g_2$ and $g_3$. On the other hand, selecting a different ray makes it difficult to treat $g_{11}$ and $g_{12}$ as equivalent. In \cite{kim2022exact}, a curved filtration approach was proposed. In this case, we address the classification problem by introducing a piecewise linear curve. In practical classification problems, filtration ray learning can also classify the time-series. By slightly perturbing the diagonal ray, it is possible to train the classifier considering that $g_{11}$ and $g_{12}$ are similar, while $g_2$ and $g_3$ are different. The goal of the following experiment is to determine whether filtration curve learning can enhance classification accuracy compared to filtration ray learning or selecting a ray with prior knowledge without filtration learning.

The experimental setup is as follows: We generated 50 time-series data for each type 

\begin{equation}
  \left\{  \begin{array}{l}  g^\epsilon_{11}  = \cos t + 1\cdot \epsilon(0, 1) \\
      g^\epsilon_{12} = \cos 2t + 1\cdot \epsilon(0, 1) \\
       g^\epsilon_{2} = 2\cos t + 1\cdot \epsilon(0, 1) \\
        g^\epsilon_{3} = 2\cos 2t + 1\cdot \epsilon(0, 1)
      \end{array}
      \right. \nonumber , 
\end{equation}
where $\epsilon(0, 1)$ represents the Gaussian standard normal errors.

The classification problem was divided into three classes, where the first class consists of $g^\epsilon_{11}$ and $g^\epsilon_{12}$, the second class consists of $g^\epsilon_{2}$ and the third class consists of $g^\epsilon_{3}$. For the experiment, 80\% of the data was used for training and the rest for testing. We use a 1-hidden layer with the width of 50 for 10,000 epochs and the resolution of the persistence image of $10^2$. The 5-fold cross-validation was used with the learning rates of $\left\{ 0.01, 0.005 \right\}$, the bandwidth of $\left\{1,0.5,0.1\right\}$ and the number of the line segments $\left\{1,2,3\right\}$.

We compare the following three methods: The first method does not involve filtration learning. It uses a fixed specific initial ray and updates only the neural network. The second method, a filtration ray learning, updates both the neural network and the direction vector. The third method, a filtration curve learning, updates both the neural network and the filtration curve. The differences in accuracy are as follows: the first method achieves an accuracy of 89.5\%, the second method achieves an accuracy of 88.5\%, and the third method achieves an accuracy of 91.5\%. These results represent the mean accuracy across 5 experiments. Figure \ref{fig:toy2} illustrates the loss and the transition of the filtration ray and filtration curve. As we can see, filtration curve learning improves classification accuracy significantly compared to filtration ray learning or selecting filtration rays using prior knowledge without filtration learning.

\begin{figure}[h]
  \centering
  \begin{minipage}{0.45\textwidth}
  \centering
    \includegraphics[width=\linewidth]{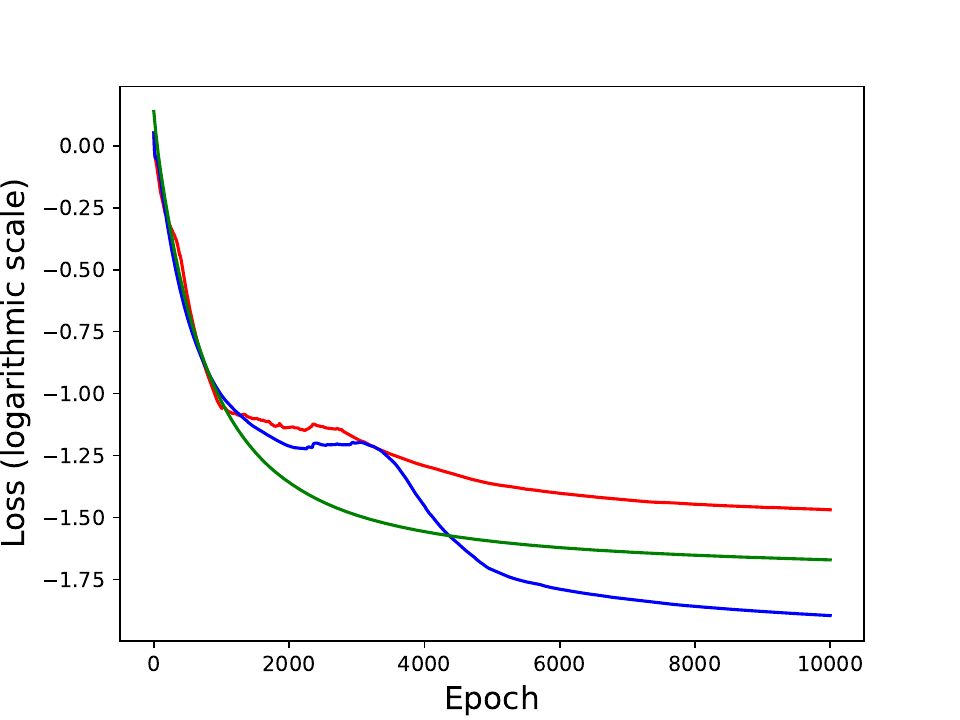}
    \end{minipage}
    \begin{minipage}{0.54\textwidth}
    \includegraphics[width=\linewidth]{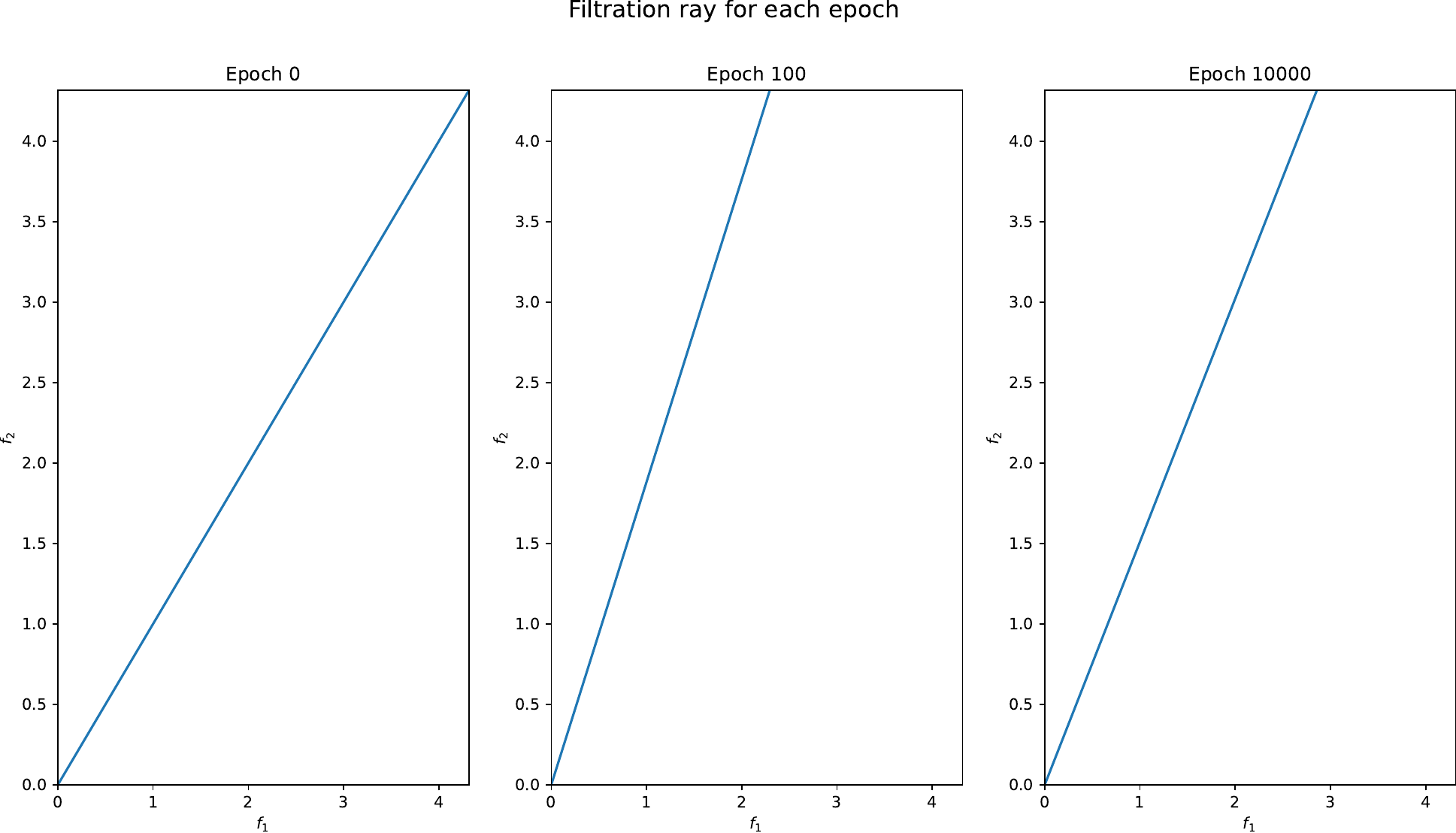}
    \end{minipage}
    \begin{minipage}{\textwidth}
      \includegraphics[width=\linewidth]{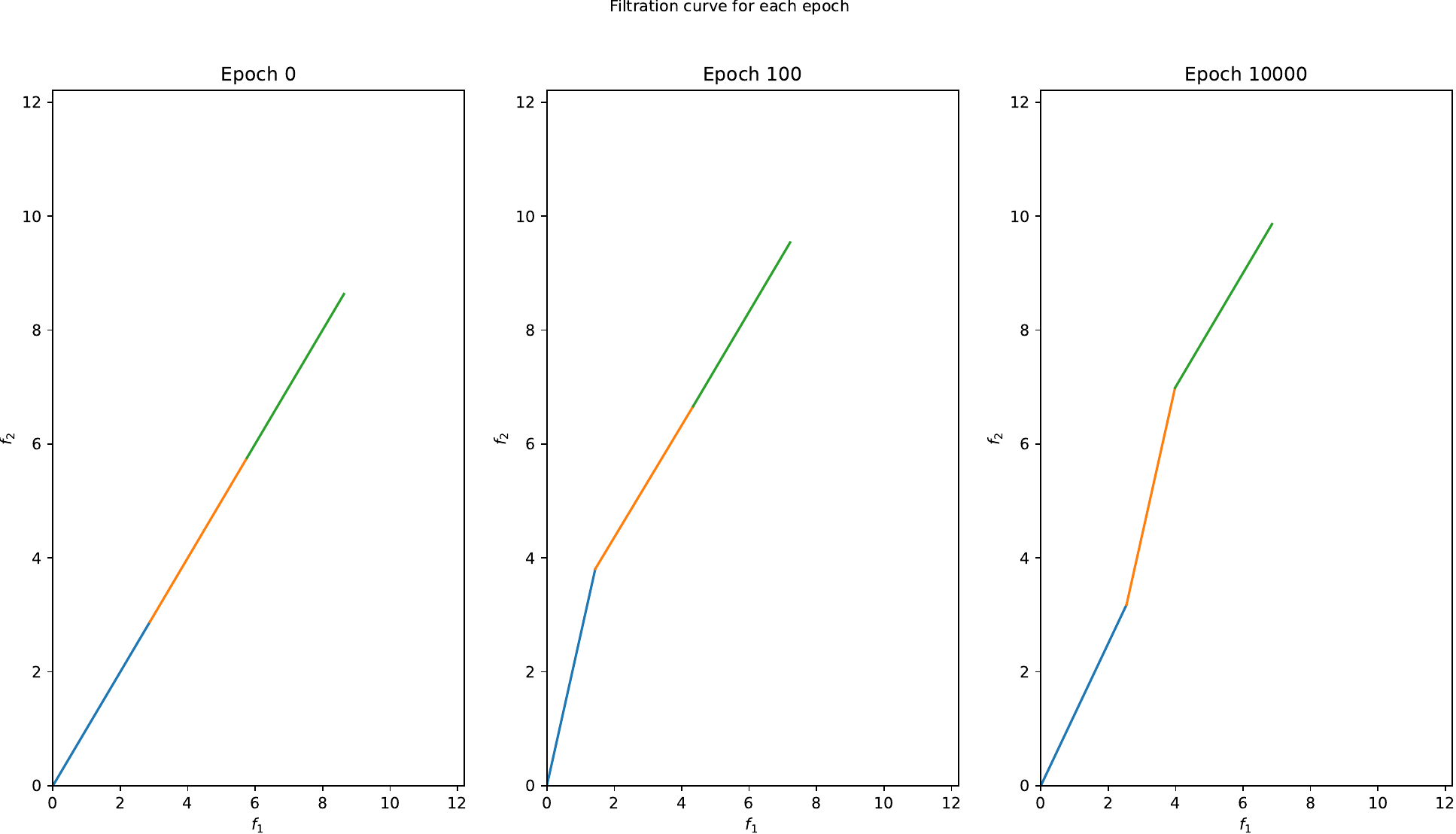}
    \end{minipage}
    \caption{Schematic illustration in Example \ref{Ex:toy2}. Top left: Loss function in logarithmic scale. Green solid line represents the case where only the neural network is updated, with a fixed curve $c(t)= (t,t)$ for $0 \le t \le \sqrt{3}$ and $c(t) = (\sqrt{3}, \sqrt{3}) + 	\left( \sqrt{{3 \over 2}}(t-\sqrt{3}), \sqrt{{1 \over 2}} (t- \sqrt{3}) \right)$ for $t \ge \sqrt{3}$. Red represents the case where both the neural network and filtration ray are updated, while blue represents the case that updates both the neural network and the three direction vectors $\a^1, \a^2$ and $\a^3$ (piecewise linear curve). Top right: Variation of the direction vector at different epochs. Bottom: Variation of the piecewise linear filtration curve with different epochs.}
    \label{fig:toy2}
\end{figure}
\end{example}

\begin{example}
\label{example:ucr}
In this example, we implement our theoretical approach on benchmarking data and evaluate its performance compared to other methods. Six datasets were selected from the UCR Time Series Classification Archive (\url{
https://www.cs.ucr.edu/~eamonn/time_series_data_2018/}, \cite{dau2019ucr}
). Our proposed model is referred to as Filtration learning within EMPH (FL-EMPH). We use two hidden layers with the widths of 50 and 10, respectively, train for 10,000 epochs and set the resolution of the persistence image to $10^2$. The 5-fold cross-validation was performed with the learning rates of $\left\{ 0.01, 0.005 \right\}$, bandwidth choices of $\left\{1,0.5,0.1\right\}$ and varying numbers of line segments in filtration curve learning to $\left\{1,2,3,5\right\}$. Table \ref{table:accuracy} shows the classification accuracy. The table provides the mean accuracy across 5 experiments. We compare our method (FL-EMPH) with TDA and other methods. These include 
\begin{enumerate}
    \item \textbf{EMPH}: The EMPH model employs topological featurization as in EMPH + PI and updates only the neural network, excluding filtration learning.
    \item \textbf{MPI}: Multi-parameter Persistence Image, \cite{carriere2020multiparameter}
    \item \textbf{PI}: Persistence Image, \cite{adams2017persistence}
    \item \textbf{MPL}: Multi-parameter Persistence Landscape, \cite{vipond2020multiparameter} 
    \item \textbf{PL}: Persistence Landscape, \cite{bubenik2015statistical}
    \item \textbf{MPK}: Multi-parameter Persistence Kernel, \cite{carriere2017sliced}
    \item \textbf{PSSK}: Persistence Scale Space Kernel, \cite{reininghaus2015stable} 
    \item \textbf{B1}: Euclidean nearest neighbor with 1-nearest neighbor classification
    \item \textbf{B2}: Dynamic time warping 1-nearest neighbor classification
    \item \textbf{B3}: Constant window width 1-nearest neighbor classification
\end{enumerate}
B1, B2 and B3 are introduced in \cite{dau2019ucr}.

\begin{table}[h]
\centering
    \includegraphics[width=1\linewidth]{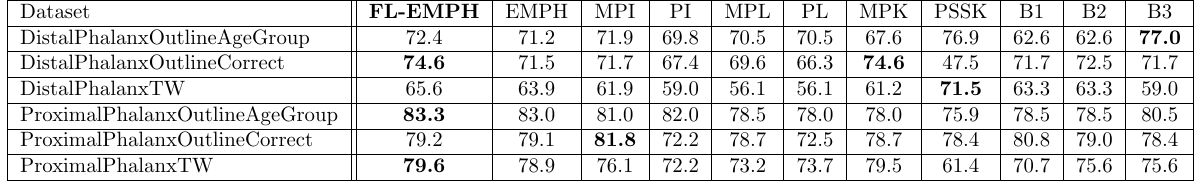}
\caption{Classification accuracy (\%). The results from MPI to B3 are cited in \cite{carriere2020multiparameter}. Those in boldface represent the best accuracy for the given dataset.}
\label{table:accuracy}
\end{table}

From Table \ref{table:accuracy}, it is evident that filtration learning effectively enhances the accuracy of classification. Moreover, when compared to other methods, the FL-EMPH approach further motivates its application to more complex datasets.
\end{example}

\begin{example}
\label{Ex:time}  
In this example, we compare the computational efficiency of updating the direction vector $\a$ with the gradient formula versus without the gradient formula. We use a dataset consisting of 400 time-series data, each with the length of 80. We employ one single hidden layer with different widths and resolutions of the persistence image. All computations were performed using an i7-13700KF CPUs, 32GB of DDR5 RAM, and an RTX 4070 Ti GPUs.  

We describe the following two models:
\begin{enumerate}
    \item Direct update of the direction vector: The first model utilizes the GPUs for both the forward and backward processes. It updates the direction vector $\a$ directly with the derivative ${\partial \mathcal{L} \over \partial \a}$ as shown in \eqref{eq:gradients}, while the neural network's parameters are updated through PyTorch's automatic differentiation.

\item Update with automatic differentiation: The second model performs both the forward and backward operations on GPUs, with every parameter being updated via PyTorch's automatic differentiation.
\end{enumerate}

Table \ref{table:computation time} shows the computation times for a total of 100 epochs for each model with different widths and resolutions. In our experiments, the direct updating model is superior in all cases, which means that using the gradient formula is beneficial for reducing computation times.

\begin{table}[h]
\centering
\renewcommand{\arraystretch}{1.5}
{\small\begin{tabular}{|m{3cm}|m{3cm}|m{3cm}|}
\hline
\multicolumn{1}{|c|}{\makecell{(Width of hidden layer, \\ Resolution)}} & \multicolumn{1}{c|}{Direct update} & \multicolumn{1}{c|}{Update w/ automatic differentiation}  \\ \hline 
\multicolumn{1}{|c|}{$(100,10^2)$} & \multicolumn{1}{c|}{\textbf{18.181 ± 0.047}} & \multicolumn{1}{c|}{70.72 ± 0.075}  \\ \hline
\multicolumn{1}{|c|}{$(100,50^2)$} & \multicolumn{1}{c|}{\textbf{19.643 ± 0.077}} & \multicolumn{1}{c|}{70.592 ± 0.315} \\ \hline
\multicolumn{1}{|c|}{$(100,100^2)$} & \multicolumn{1}{c|}{\textbf{27.209 ± 0.211}} & \multicolumn{1}{c|}{131.672 ± 0.04}  \\ \hline
\multicolumn{1}{|c|}{$(1000,10^2)$} & \multicolumn{1}{c|}{\textbf{18.757 ± 0.061}} & \multicolumn{1}{c|}{70.005 ± 0.06}  \\ \hline
\multicolumn{1}{|c|}{$(1000,50^2)$} & \multicolumn{1}{c|}{\textbf{19.679 ± 0.063}} & \multicolumn{1}{c|}{69.969 ± 0.125}  \\ \hline
\multicolumn{1}{|c|}{$(100,100^2)$} & \multicolumn{1}{c|}{\textbf{29.878 ± 0.04}} & \multicolumn{1}{c|}{131.843 ± 0.065}  \\ \hline
\multicolumn{1}{|c|}{$(10000,10^2)$} & \multicolumn{1}{c|}{\textbf{19.078 ± 0.069}} & \multicolumn{1}{c|}{69.858 ± 0.036}  \\ \hline
\multicolumn{1}{|c|}{$(10000,50^2)$} & \multicolumn{1}{c|}{\textbf{27.495 ± 0.037}} & \multicolumn{1}{c|}{70.43 ± 0.042}  \\ \hline
\multicolumn{1}{|c|}{$(10000,100^2)$} & \multicolumn{1}{c|}{\textbf{85.928 ± 0.028}} & \multicolumn{1}{c|}{153.225 ± 0.041} \\ \hline
\end{tabular}}
\caption{Average computation times (second) for 100 epochs of filtration ray learning. The shortest time is highlighted in a boldface.}
\label{table:computation time}
\end{table}

\end{example}
\end{section}

\begin{section}{Conclusion}
\label{sec:conclusion}
In this paper, we proposed an optimization problem to find an optimal filtration curve within EMPH for the classification of time-series data. Our experiments demonstrate that focusing on finding an optimal filtration curve is meaningful in terms of classification accuracy, and generalizing in a higher-dimensional parameter space is easier than considering multiple lines within the parameter space. Moreover, filtration learning involves updating filtration, calculating the persistence image and updating neural network for each epoch; this process is computationally intensive. However, using the exact formula of the barcode, we can deduce the gradient ${\partial \mathcal{L} \over \partial \a}$. Therefore, we can directly update $\a$ without automatic differentiation. Our approach effectively reduces the computational complexity of both forward and backward steps, facilitating efficient filtration learning. However, considering the still high computational costs, combining our model with a method to reduce the dimensionality of the multi-parameter space seems to be a promising approach. Our model can enhance its performance by incorporating various techniques from deep learning such as minibatch learning, dropout and early stopping, which we will apply to real-world data. Additionally, although our experiments so far have utilized only one-dimensional persistent homology, we will extend our experiment to include higher-dimensional persistent homology.
\end{section}

\vskip .1in 
{\bf Acknowledgement}
This work was supported by NRF of Korea under the grant number 2021R1A2C3009648 and POSTECH Basic Science Research Institute under the NRF grant number 2021R1A6A1A1004294412. This work is also supported partially by NRF grant by the Korea government (MSIT) (RS-2023-00219980).

\appendix

\section{o-minimal structure}
\label{appendix:o-minimal}
This section is based on \cite{van1998tame, coste2000introduction, loi2003tame, davis2020stochastic}.

The motivation of o-minimal structure is consider the class of nice subsets of $\R^n$ and class of nice functions between the nice set. Nice means the class closed under elementary logical (unions, intersections, complement), geometric and topological operation (closure, cartesian products, project into lower dimensional subspace). One of the example of the class of nice set is semialgebraic sets. In the name o-minimal, `o' stands for order.

\begin{definition}[Semialgebraic set, Semialgebraic function]
A subset $X \subseteq \R^n$ is semialgebraic if $X$ is a finite union of $\left\{ x \in \R^n : P(x) =0 \text{ and } Q_1(x), \cdots, Q_l(x)>0 \right\}$, where $l \in \mathbb{Z}_{\ge 0}$ and $P, Q_1 , \cdots,Q_l \in \R[x_1, \cdots, x_n]$. For semialgebraic sets $A \subseteq \R^m$ and $B \subseteq \R^n$, a function $f : A \rightarrow B$ is called semialgebraic if its graph is a semialgebraic subset of $\R^m \times \R^n$.     
\end{definition}

Boolean algebra is an axiomatic extension of power set.
\begin{definition}[Boolean algebra, \cite{halmos2009introduction}]
Boolean algebra is a nonempty set $\mathcal{A}$ with binary operations $\lor$ and $\land$ on $\mathcal{A}$ and a unary operation $'$ and two distinguished elements $0$ and $1$ in $\mathcal{A}$ satisfying suitable operation rules such as in a power set of nonempty set.
\begin{table}[h]
\centering
\begin{tabular}{lll}
$\mathcal{A}$ & $\longleftrightarrow$ & $\mathcal{P}(X)$ \\
$\lor$        & $\longleftrightarrow$ & $\bigcup$                       \\
$\land$       & $\longleftrightarrow$ & $\bigcap$                       \\
$'$           & $\longleftrightarrow$ & $^c$                            \\
$0$           & $\longleftrightarrow$ & $\emptyset$                     \\
$1$           & $\longleftrightarrow$ & $X$                            
\end{tabular}
\caption{Correspondence between Boolean algebra and a power set of nonempty set.}
\end{table}
\end{definition}

o-minimal structure is an axiomatic extension of semialgebraic set.

\begin{definition}
o-minimal structure on the real number field $(\R, +, \cdot)$ is a sequence of $\mathcal{D} = (\mathcal{D}_n)_{n \in \mathbb{N}}$ such that
\begin{enumerate}
    \item $\mathcal{D}_n$ is a boolean algebra of subsets of $\R^n$. i.e.,$\emptyset, \R^n \in \mathcal{D}_n$, and if $A,B \in \mathcal{D}_n$, then $A \cup B \in \mathcal{D}_n$ and $\R^n \setminus A \in \mathcal{D}_n$.
    \item If $A \in \mathcal{D}_n$, then $A \times \R$ and $\R \times A \in \mathcal{D}_{n+1}$.
    \item If $A \in \mathcal{D}_{n+1}$, then $\pi(A) \in \mathcal{D}_n$, where $\pi: \R^{n} \times \R \rightarrow \R^n$ is the projection on the first $n$ coordinates.
    \item $\mathcal{D}_n$ contains $\left\{ x \in \R^n : p(x)=0 \right\}$ for every polynomial $p \in \R[x_1, \cdots,x_n]$.
    \item The elements of $\mathcal{D}_1$ are exactly the finite unions of intervals and points. (This is called o-minimal constraint)
\end{enumerate}
\end{definition}

\begin{definition}
We say $A \subseteq \R^m$ is definable set if $A \in \mathcal{D}_n$. For a definable set $A$, a function $f : A \rightarrow \R^n$ is called definable function if its graph is definable set in $\R^{m+n}$.    
\end{definition}

\begin{proposition}
Suppose $A\subseteq \R^m$ is definable set and $f : A \rightarrow \R^n$ is a definable function. Then the followings are hold. 

\begin{enumerate}
    \item Let $A \in \mathcal{D}_m$ and $B \in \mathcal{D}_n$. Then closure, interior of $A$ is also definable, i.e., $cl(A), int(A) \in \mathcal{D}_m$. Additionally, $A \times B \in \mathcal{D}_{m+n}$ also holds \cite{van1998tame}.
    \item Sum, product, composition, restriction and inverse of definable functions is also definable. Image of definable function is a definable set. Moreover, $f=(f_1, \cdots, f_n)$ is definable if and only if $f_i$ are definable \cite{van1998tame}.
    \item Each definable set $A \in \R^m$ has a finite partition $A = C_1 \cup \cdots \cup C_k$ into $C^1$-cells and $f|_{\mathcal{C}_i}$ are $C^1$ \cite{van1998tame}.
    \item There is an o-minimal structure that contains the graph of the exponential function and semialgebraic sets \cite{davis2020stochastic, wilkie1996model}.
\end{enumerate}
\end{proposition}

\begin{proposition}[\cite{ta2011introduction}]
\label{prop:semi properties}
Let $f: X \rightarrow \mathbb{R}$ be a semialgebraic function. If $f(x) \ne 0$ for every $x \in X$, then ${1 \over f}$ is semialgebraic, and if $f \ge 0$, then $\sqrt{f}$ is also a semialgebraic function.   
\end{proposition}

\begin{proposition}[\cite{ta2011introduction}]
\label{prop:min semi}
If $f$ and $g$ are semialgebraic, then $\min\left\{ f,g \right\}$ is also semialgebraic. 
\end{proposition}

\begin{proposition}[\cite{loi2003tame}]
\label{prop:semi-definable}
Every semialgebraic set is definable in any o-minimal structure.    
\end{proposition}

\section{Propositions for showing $\mathcal{L}$ is locally Lipschitz}
\label{appendix:propositions}
\begin{proposition}
For locally Lipschitz functions $f : \mathbb{R}^n \rightarrow \mathbb{R}^m$ and $g : \mathbb{R}^m \rightarrow \mathbb{R}$, $g \circ f$ is also locally Lipschitz.
\end{proposition}

\begin{proof}
Let $a = f(x)$. Since $g$ is locally Lipschitz, there is a neighborhood $U_{a}$ of $a$ such that $g|_{U_{a}}$ is Lipschitz. i.e., there is $L_1 \in \mathbb{R}$ such that for every $b \in U_a$, $\left| g(a) - g(b) \right| \le L_1 \left| a - b \right|$. Given that $f$ is continuous, there is a neighborhood $V_x$ of $x$ such that $f(V_x) \subseteq U_{a}$. Since $f$ is locally Lipschitz, there is a $L_2 \in \mathbb{R}$ and a neighborhood $W_x$ such that $\left| f(x) - f(y) \right| \le L_2 \left| x-y \right|$ for every $y \in W_x$. Consequently, for any $y \in V_x \bigcap W_x$, it follows that $\left| g(f(x)) - g(f(y)) \right| \le L_1 \left| f(x) - f(y) \right| \le L_1L_2 \left| x - y \right|$.
\end{proof}

\begin{proposition}[Exercise 4.2.10, \cite{kumaresan2005topology}]
\label{prop:locally Lipschitz}
If $X$ is a locally compact metric space, then $f$ is locally Lipschitz if and only if it is Lipschitz continuous on every compact subset of $X$.    
\end{proposition}

\begin{proposition}
\label{prop:locally Lipschitz2}
If $f_1, f_2$ are locally Lipschitz, then $\min \left\{ f_1, f_2 \right\}$ and $\max \left\{ f_1, f_2 \right\}$ are locally Lipschitz. 
\end{proposition}

\begin{proof}
Since the sum of locally Lipschitz is locally Lipschitz and $\left| f_1 - f_2 \right|$ is locally Lipschitz, $\min\left\{ f_1, f_2 \right\} = {f_1 + f_2 - \left| f_1 - f_2 \right| \over 2}$ is locally Lipschitz. Maximum function case can be proved similarly.   
\end{proof}

\section{Vectorization of Multi-parameter persistent homology}
\label{app:vectorization}

In this section, we review the vectorization of multi-parameter persistent homology, using multi-parameter persistence image and landscape as examples, to compare our method. The following definitions are based on the Python code  \footnote{\url{https://github.com/MathieuCarriere/multipers}, \cite{carriere2020multiparameter}} \footnote{\url{https://github.com/OliverVipond/Multiparameter_Persistence_Landscapes}, \cite{vipond2020multiparameter}}.


\begin{definition}[Two-parameter persistence image, \cite{carriere2020multiparameter}]
\label{def:MPI}
Consider a grid $\mathcal{G} = \left\{\mathbf{g}_{i,j} \right\}_{0\le i,j \le r-1}$, where each $\mathbf{g}_{i,j}$ is defined as $\left(m_1 + i \frac{M_1-m_1}{r-1}, m_2 + j \frac{M_2-m_2}{r-1}\right)$. Denote $\ell(x,v)$ as a ray passing through $x$ with direction vector $v$. Let $L = \biggl\{ \ell	\left( (m_1 + i \delta, m_2 ) , \left({\sqrt{2} \over 2}, {\sqrt{2} \over 2}\right) \right) : 0 \le i \le \lfloor {\# \text{rays} \over 2} \rfloor \biggr\} \bigcup \left\{ \ell	\left( (m_1, m_2 + j \delta ) , \left({\sqrt{2} \over 2}, {\sqrt{2} \over 2}\right) \right): 0 \le j \le \lfloor {\# \text{rays} \over 2} \rfloor \right\}$ be a collection of rays, where $\delta = {(M_1-m_1) + (M_2-m_2) \over \# \text{rays}}$. For a Two-parameter persistent homology $\mathsf{MP}$, we construct an $r^2$-dimensional vector $\mathcal{I} = (\mathcal{I}_{i,j})_{0 \le i,j \le r-1}$ where
$$\mathcal{I}_{i,j} = \sum\limits_{I \in D_L(\mathsf{MP})} w_1 (I) \left( w_2(\ell^*) \exp \left( - { \underset{\ell \in I}{\min} \lVert \mathbf{g}_{i,j}, \ell \rVert^2 \over \sigma^2} \right) \right),$$ and in this context, $D_L(\mathsf{MP})$ is the vineyard decomposition of the fibered barcodes $\left\{ \mathsf{MP}^{\ell_{i}} \right\}$, $\lVert \mathbf{g}_{i,j}, \ell \rVert$ is defined as $\underset{x \in \ell}{\min} \lVert \mathbf{g}_{i,j} - x \rVert$. Here, $\ell^*$ denotes the ray for which $\underset{\ell \in I}{\min} \lVert \mathbf{g}_{i,j}, \ell \rVert^2$ is achieved, $A(I)$ represents the area of the convex hull formed by the endpoints of all bars in $I$, the function $w_1 : D_L(\mathsf{MP}) \rightarrow \mathbb{R}$ is given by $w_1(I) = \left( \frac{A(I)}{(M_1-m_1)(M_2-m_2)} \right)^q$, and $w_2(\ell)$ is $\min(e^\ell_1 , e^\ell_2)$, where $(e^\ell_1 , e^\ell_2)$ is the direction vector of $\ell$.
   
\end{definition}

\begin{definition}[Two-parameter persistence landscape, \cite{vipond2020multiparameter}]
\label{def:MPL}
All notations are shared in Definition \ref{def:MPI}. Denote rays by
$$\ell_i = \begin{cases}
\ell \left( (m_1 - i {M_1-m_1 \over r-1}, m_2 ) , \left(1,1 \right) \right), & \mbox{if }-(r-1) \le i \le 0 \\
\ell	\Bigl( (m_1, m_2 + i{M_2-m_2 \over r-1} ) , \left(1,1 \right) \Bigr), & \mbox{if } 0< i \le r-1
\end{cases},$$
and endpoint of rays $l_i$ by
$$\o_i = \begin{cases}
(m_1 - i {M_1-m_1 \over r-1}, m_2 ), & \mbox{if }-(r-1) \le i \le 0 \\
(m_1, m_2 + i{M_2-m_2 \over r-1} ), & \mbox{if } 0< i \le r-1
\end{cases}.$$
Define a $k$-th two-parameter persistence landscape $\mathcal{S} = (\mathcal{S}_{i,j})_{0 \le i,j \le r-1} \in \R^{r^2}$, where 

$$\mathcal{S}_{i,j} = k\text{-th largest value of } \bigcup\limits_{s}	\left\{\max \left\{ 0, \min \left\{  {\lVert \mathbf{g}_{i,j} - \o_{j-i} \rVert  - b_s \over \sqrt{2}} , {d_s - \lVert \mathbf{g}_{i,j} - \o_{j-i} \rVert  \over \sqrt{2}} \right\} \right\}\right\},$$
where $(b_s, d_s] \in \mathsf{bcd}^{\ell_{j-i}}_{*}(\mathsf{MP})$.
\end{definition}

Although Definition \ref{def:MPI} and \ref{def:MPL} may seem complicated, we have prepared the following simple example.

\begin{example}
\label{ex:MPI}
Let $\mathcal{G}$ be a grid with $m_1=m_2=0$,$M_1=M_2=2$ and $r=2$ in Definition \ref{def:MPI}. Take the graph $X = \left\{ a,b,c,d,ac,bc \right\}$ and define a two-parameter filtration $F : X \rightarrow \R^2$ with $F(a) = (0,1), F(b) = (1,0), F(c) = (1,1), F(d) = (1,1), F(ac) = (1,1)$, and $F(bc)= (1,1)$. Define a two-parameter persistent homology $\mathsf{MP}$ as $\mathsf{MP}(r_1,r_2):= H_0 	\left(F^{-1}(r_1, r_2)\right)$.

\begin{enumerate}
    \item (Two-parameter persistence image)
Set $q=1$ and $\sigma=1$ in Definition \ref{def:MPI}. The Vineyard Decomposition $D_L(\mathsf{MP}) = \left\{ I_1, I_2 \right\}$, where $I_1$ consists of three blue lines and $I_2$ consists of a single red line as shown in Figure \ref{fig:example MPI}. For the components of the two-parameter persistence image, the calculations are as follows: $A(I_1) = {20 \over 9}, A(I_2) = 0, w_1(I_1) = {5 \over 9}, w_1(I_2) = 0$, and $w_2(\ell^*) = {\sqrt{2} \over 2}$. 
\begin{align*}
\mathcal{I}_{00} &= {5 \over 9} \cdot {\sqrt{2} \over 2} \cdot \exp\left(- {16 \over 9} \right) \approx 0.07, \\
\mathcal{I}_{01} &= {5 \over 9} \cdot {\sqrt{2} \over 2} \cdot \exp\left(- {2 \over 9} \right) \approx 0.31, \\
\mathcal{I}_{10} &= {5 \over 9} \cdot {\sqrt{2} \over 2} \cdot \exp\left(- {2 \over 9} \right) \approx 0.31, \\
\mathcal{I}_{11} &= {5 \over 9} \cdot {\sqrt{2} \over 2} \cdot \exp\left(0 \right) \approx 0.39.
\end{align*}
\begin{figure}[h]
    \centering
    \includegraphics[width=1\linewidth]{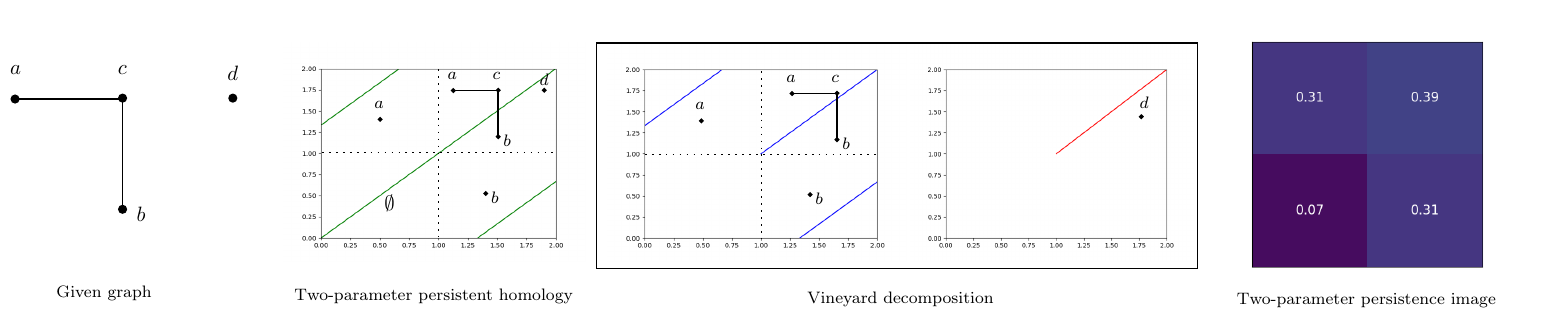}
    \caption{Supplementary of two-parameter persistence image.}
    \label{fig:example MPI}
\end{figure}

\item (Two-parameter persistence landscape)
Set $k=1$ in Definition \ref{def:MPL}. Then $\ell_{-1} = \ell\left( (2,0), (1,1) \right), \ell_{0} = \ell\left( (0,0), (1,1) \right)$ and $\ell_{1} = \ell\left( (0,2), (1,1) \right)$. The fibered barcodes are $\mathsf{bcd}^{\ell_{-1}}_{0}(\mathsf{MP}) = 	\left\{ (0, \infty), (\sqrt{2}, \infty) \right\},$
$\mathsf{bcd}^{\ell_{0}}_{0}(\mathsf{MP}) = 	\left\{(\sqrt{2}, \infty)_{(2)} \right\}$ and $\mathsf{bcd}^{\ell_{1}}_{0}(\mathsf{MP}) = 	\left\{ (0, \infty), (\sqrt{2}, \infty) \right\}$. The distances $\lVert \mathbf{g}_{0,0} - \o_0 \rVert = \lVert \mathbf{g}_{1,0} - \o_{-1} \rVert = \lVert \mathbf{g}_{0,1} - \o_{1} \rVert = 0$ and $\lVert \mathbf{g}_{1,1} - \o_0 \rVert = 2 \sqrt{2}$.

{\scriptsize
\begin{align*}
\mathcal{S}_{00} &= 1\text{-st largest value of } \left\{\max \left\{ 0, \min	\left\{{0-\sqrt{2} \over \sqrt{2}}, \infty \right\}\right\},  \max \left\{ 0, \min	\left\{{0-\sqrt{2} \over \sqrt{2}}, \infty \right\} \right\}\right\} = 0, \\
\mathcal{S}_{01} &= 1\text{-st largest value of } 	\left\{\max \left\{ 0, \min	\left\{{0-0 \over \sqrt{2}}, \infty \right\}\right\},  \max \left\{ 0, \min	\left\{{0-\sqrt{2} \over \sqrt{2}}, \infty \right\} \right\}\right\} = 0, \\
\mathcal{S}_{10} &= 1\text{-st largest value of } 	\left\{\max \left\{ 0, \min	\left\{{0-0 \over \sqrt{2}}, \infty \right\}\right\},  \max \left\{ 0, \min	\left\{{0-\sqrt{2} \over \sqrt{2}}, \infty \right\} \right\}\right\} = 0, \\
\mathcal{S}_{11} &= 1\text{-st largest value of } 	\left\{\max \left\{ 0, \min	\left\{{2\sqrt{2}-\sqrt{2} \over \sqrt{2}}, \infty \right\}\right\},  \max \left\{ 0, \min	\left\{{2\sqrt{2}-\sqrt{2} \over \sqrt{2}}, \infty \right\} \right\}\right\}  = 1.
\end{align*}
}
\begin{figure}[h]
    \centering
    \includegraphics[width=1\linewidth]{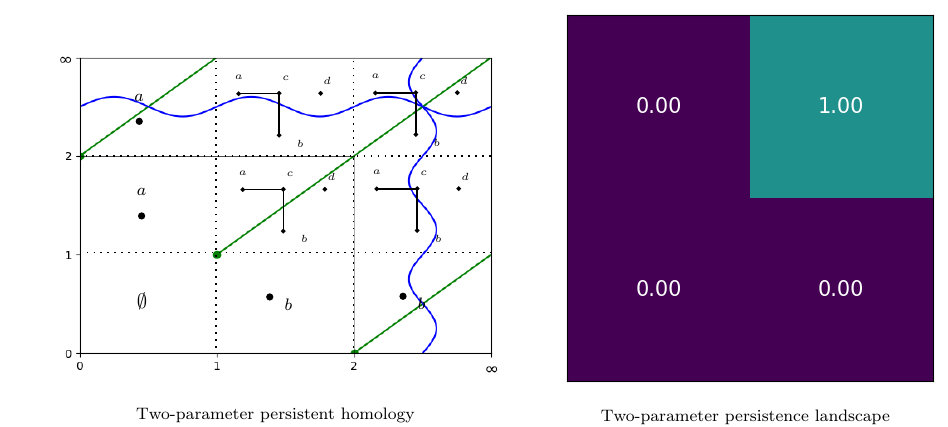}
    \caption{Supplementary of two-parameter persistence landscape.}
    \label{fig:example MPL}
\end{figure}
\end{enumerate}
\end{example}

\vskip .1in
\normalem
\bibliographystyle{ieeetr}
\nocite*{}
\bibliography{reference}

\end{document}